\documentclass[11pt]{article}

\usepackage{geometry}
\usepackage{amscd}                		
\usepackage{graphicx}			
\usepackage{amssymb}
\usepackage{amsmath}
\usepackage{amsthm}
\usepackage{amsfonts}
\usepackage{pb-diagram}
\usepackage{array}
\usepackage{multirow}
\usepackage{dcolumn}
\usepackage{enumerate}
\usepackage{comment}
\usepackage{multicol}
\usepackage{mathrsfs}
\usepackage[abbrev]{amsrefs}
\usepackage{braket}

\theoremstyle{plain}
\newtheorem{thm}{Theorem}[subsection]
\newtheorem{lem}[thm]{Lemma}

\newtheorem{prop}[thm]{Proposition}
\newtheorem{cor}[thm]{Corollary}

\theoremstyle{definition}
\newtheorem{dfn}[thm]{Definition}

\newtheorem{ex}[thm]{Example}

\theoremstyle{remark}
\newtheorem{rmk}{Remark}

\newcommand{\Hom}{\mathop{\mathrm{Hom}}\nolimits}
\newcommand{\End}{\mathop{\mathrm{End}}\nolimits}
\newcommand{\Aut}{\mathop{\mathrm{Aut}}\nolimits}

\newcommand{\Res}{\mathop{\mathrm{Res}}\nolimits}
\newcommand{\Stab}{\mathop{\mathrm{Stab}}\nolimits}

\newcommand{\reg}{\mathop{\mathrm{reg}}\nolimits}
\newcommand{\Ann}{\mathop{\mathrm{Ann}}\nolimits}

\newcommand{\re}[1]{\operatorname{Re}(#1)}
\newcommand{\im}[1]{\operatorname{Im}(#1)}

\newcommand*{\data}[2]{#1\hookrightarrow #2}
\newcommand{\Det}{\mathop{\mathrm{Det}}\nolimits}
\newcommand{\bs}{\backslash}
\newcommand{\dtimes}{d^{\times}\!}

\def\hsymb#1{\mbox{\strut\rlap{\smash{\Large$#1$}}\quad}}

\numberwithin{equation}{section}

\title{Relative Hecke's integral formula for an arbitrary extension of number fields}
\author{Hohto Bekki}
\date{}

\begin{document}
\maketitle
\begin{abstract}
In this article, we present a generalized Hecke's integral formula for an arbitrary extension $E/F$ of number fields. 
As an application, we present relative versions of the residue formula and Kronecker's limit formula for the ``relative" partial zeta function of $E/F$. 
This gives a simultaneous generalization of two different known results given by Hecke himself and Yamamoto. 
\end{abstract}

\section{Introduction}\label{intro}
Let $k$ be a number field (of finite degree), and $\mathcal O_k$ be the ring of integers of $k$. Let $\mathscr A$ be an ideal class of $k$, and let $\mathfrak a \in \mathscr A$. Then the partial zeta function $\zeta_k(\mathscr A, s)$ = $\zeta_k(\mathfrak a,s)$ associated to the ideal class $\mathscr A$ (or to the ideal $\mathfrak a$) is defined as,

\begin{equation}\label{partial zeta}
\zeta_k(\mathscr A,s):= \sum_{\substack{\mathfrak b \in \mathscr A\\ \mathfrak b \subset \mathcal O_k}}\frac{1}{N\mathfrak b^{s}} \quad(\re s >1).
\end{equation}

In the case where $k$ is real quadratic, the classical Hecke's integral formula expresses this partial zeta function as the integral of the real analytic Eisenstein series along the closed geodesic on the modular curve $SL_2(\mathbb Z) \backslash \mathfrak h$ associated to the ideal class $\mathscr A$, where $\mathfrak h :=\{ z\ \in \mathbb C \mid \im z >0\}$ is the Poincar\'e upper half plane. 

To be precise, let $k$ be real quadratic, and fix an embedding $k \hookrightarrow \mathbb R$. Suppose $\mathfrak a \in \mathscr A$ is taken to be of the form $\mathfrak a =\mathbb Z +\mathbb Z \alpha \subset k$, where $\alpha \in k$ is a real quadratic irrational. Let $\bar{\alpha}$ be the conjugate of $\alpha$ over $\mathbb Q$, and let $\varpi$ be the geodesic on $\mathfrak h$ connecting $\alpha$ and $\bar{\alpha}$:
\begin{equation}\label{closed geodesic}
\varpi : \mathbb R_{>0} \rightarrow \mathfrak h;~ t \mapsto \frac{\alpha ti+ \bar{\alpha}t^{-1}}{ti+t^{-1}}.
\end{equation}
Then it is known that $\varpi$ projected to $SL_2(\mathbb Z) \backslash \mathfrak h$ becomes periodic. More precisely, $\varpi$ induces a closed geodesic
\begin{equation}
\overline{\varpi} : \mathbb R_{>0}/ \varepsilon^{2\mathbb Z} \rightarrow SL_2(\mathbb Z)\backslash \mathfrak h,
\end{equation}
where $\varepsilon>1$ is a fundamental unit of $\mathcal O_k$ (cf. \cite{bekki17}, \cite{sarnak82}).

Now, let $E(z,s)$ be the real analytic Eisenstein series defined by
\begin{equation}\label{class eis ser}
E(z,s):= \frac{1}{2}\sum_{\substack{(c,d) \in \mathbb Z^2 \\ (c,d)=1}} \frac{ \im z^s}{|cz+d|^{2s}}, \quad\text{for $z \in \mathfrak h$ and $s \in \mathbb C, \re s>1$}, 
\end{equation}
which is well-defined on $SL_2(\mathbb Z)\backslash \mathfrak h$. Then we have the following Hecke's integral formula.

\begin{thm}[\cite{hecke17}]\label{classical hecke's int formula} We have
\begin{equation}
\int_{ \mathbb R_{>0}/ \varepsilon^{2\mathbb Z}} E(\overline{\varpi}(t),s) \frac{dt}{t} = \frac{1}{2}d_k^{s/2} \frac{\Gamma (s/2)^2}{\Gamma (s)} \frac{\zeta_k(\mathfrak a^{-1}, s)}{\zeta_{\mathbb Q}(2s)},
\end{equation}
where $d_k$ is the discriminant of $k$.
\end{thm}

Many authors including Hecke himself have studied generalizations of this formula. Hecke~\cite{hecke17} generalizes the formula to the case of an arbitrary number field, and Hiroe-Oda~\cite{hiroeoda08} extend Hecke's result to $L$-functions twisted by Grossencharacters. Another generalization is obtained by Yamamoto~\cite{yamamoto08}, who generalizes the formula to the case of an arbitrary quadratic extension $E/F$ of number fields. 
In the following we refer to any generalization of Theorem \ref{classical hecke's int formula} as Hecke's integral formula.

In a previous paper \cite{bekki17}, motivated by the analogy between the above periodicity of the geodesic $\varpi$ and the classical Lagrange's theorem in the theory of continued fractions, we  have considered generalization of closed geodesics in the symmetric space for $GL_n$. As a result, we have established some new geodesic multi-dimensional continued fraction algorithms, and have proved generalizations of Lagrange's theorem. In this paper, using the same idea as in \cite{bekki17}, we present Hecke's integral formula for an arbitrary extension $E/F$ of number fields (Theorem \ref{hecke's int formula}). Hecke's result corresponds to the case where $F=\mathbb Q$, and Yamamoto's result corresponds to the case where $[E:F]=2$.
In our argument, we are naturally led to introduce a ``relative" partial zeta function $\zeta_{E/F, \mathscr A}(A, s)$  ($\mathscr A \in Cl_F, A \in Cl_E$) (Definition \ref{relative partial zeta}), which gives a ``decomposition of the partial zeta function $\zeta_E(A,s)$ along $Cl_F$'':
\begin{align}
\zeta_E(A, s)= \sum_{\mathscr A \in Cl_F} \zeta_{E/F, \mathscr A}(A, s).
\end{align}  
As an application of our Hecke's integral formula, we obtain the residue formula and Kronecker's limit formula for this relative partial zeta function $\zeta_{E/F, \mathscr A}(A, s)$, that is, formulas for the residue and the constant term of $\zeta_{E/F, \mathscr A}(A, s)$ at $s=1$ (Theorem \ref{rel partial zeta thm}). 
The author thinks it is interesting that both the special value of $\zeta_E(A,s)$ at $s=1$ and the special value of $\zeta_F(\mathscr A^{-1},s)$ {\bf at} {\boldmath $s=n$} appear simultaneously in the residue formula (\ref{rel zeta res formula}) of the relative partial zeta function $\zeta_{E/F, \mathscr A}(A,s)$. As far as the author is aware, such a phenomenon has not been observed in the previous works. 

There are also many preceding works on Kronecker's limit formula for the zeta functions of number fields. Hecke remarks in \cite{hecke17} that one can deduce Kronecker's limit formula for general number fields from the result of Epstein~\cite[p.~644]{epstein03}. Liu and Masri~\cite{liumasri15} use this formula for totally real fields to obtain an analogue of Kronecker's solution of Pell's equation. Bump and Goldfeld~\cite{bumpgoldfeld84} give a different proof in the case of totally real cubic fields. The case of relative quadratic extensions of number fields is obtained by Yamamoto~\cite{yamamoto08}, and our result generalizes all of these results.

\paragraph{Outline of this paper}
In Section \ref{setting}, we set up generalizations of the Poincar\'e upper half plane $\mathfrak h$, the above geodesic $\varpi$ on $\mathfrak h$, and the real analytic Eisenstein series $E(z,s)$. We use the symmetric space for $\Res_{F/\mathbb Q}GL_n$ as a generalization of $\mathfrak h$ where $F$ is  a number field (cf. Borel~\cite{borel77}), and define a certain totally geodesic submanifold called the Heegner object in the symmetric space which plays a part of the above geodesic $\varpi$. Then we consider the Eisenstein series for $\Res_{F/\mathbb Q}GL_n$ following the general construction of Langlands~\cite{langlands76}. 
In Section \ref{integral formula}, we prove our first main theorem: Theorem \ref{hecke's int formula}. 

Sections \ref{fourier exp} and \ref{res and limit formula} are devoted to prove the residue formula and Kronecker's limit formula for the relative partial zeta function $\zeta_{E/F, \mathscr A}(A, s)$.
In Section \ref{fourier exp}, we compute the Fourier expansion of the Eisenstein series as a preparation.
In Section \ref{res and limit formula}, we first prove the residue formula and Kronecker's limit formula for the Eisenstein series using the Fourier expansion. Then using Theorem \ref{hecke's int formula}, we obtain our second main theorem: Theorem \ref{rel partial zeta thm}.


\paragraph{Some remarks on Eisenstein series}
In this paper, we construct Eisenstein series by generalizing the argument of Goldfeld~\cite{goldfeld06}. This construction fits into the general theory of Eisenstein series established by Langlands~\cite{langlands76}. 
Therefore some of the results in this paper concerning the Eisenstein series (e.g., the convergence, analytic continuation, and some part of Fourier coefficients) might be obtained directly from the general theory. However, as far as the author is aware, our Eisenstein series $E_{L, \mathscr A}(z,s)$ ($\mathscr A \in Cl_F$) (Definition \ref{var eis ser}) has not been studied well, while the sum $E_{L}(z,s) = \sum_{\mathscr A \in Cl_F}E_{L, \mathscr A}(z,s)$ is a generalization of those traditional Eisenstein series studied by Asai~\cite{asai70}, Jorgenson-Lang~\cite{jorgenson-lang99}, Yoshida~\cite{yoshida03}, and Yamamoto~\cite{yamamoto08} in the case where $n=2$. 
In Sections \ref{fourier exp} and \ref{res and limit formula}, we give an explicit Fourier expansion formula, residue formula, and Kronecker's limit formula for our Eisenstein series $E_{L, \mathscr A}(z,s)$, which gives a generalization of the corresponding results proved in \cite{asai70}, \cite{jorgenson-lang99}, \cite{yamamoto08} and \cite{liumasri15}.


\subsection{Convention}\label{sect convention}

Let $\Lambda$ be any index set, and let $\{X_{\lambda}\}_{\lambda \in \Lambda}$ be any family of sets indexed by $\Lambda$. For $x \in \prod_{\lambda \in \Lambda}X_{\lambda}$, we often denote by $x_{\lambda} \in X_{\lambda}$ the $\lambda$-component of $x$ without specifying.

In this paper, a number field is always assumed to be of finite degree over the field $\mathbb Q$ of rational numbers. 
For a number field $k$, we denote by $S_k$ the set of archimedean places of $k$. For $\sigma \in S_k$, we denote by $k_{\sigma}$ the completion of $k$ at $\sigma$, and denote by $n_{\sigma}:=[k_{\sigma}:\mathbb R]$ the local extension degree. We also denote by $\sigma :k \hookrightarrow k_{\sigma}$ the completion map. We denote by $k_{\infty}:= k \otimes_{\mathbb Q}\mathbb R$ the infinite adele of $k$. Then we have $k_{\infty} \simeq \prod_{\sigma \in S_k}k_{\sigma}$. The number of archimedean (resp. real, complex) places is denoted by $r_k$ (resp. $r_1(k), r_2(k)$). 
As usual, we denote by $\mathcal O_k$ the ring of integers, $\mathcal O_k^{\times}$ its group of units, $Cl_k$ the ideal class group, and $h_k$ the class number. 
We denote by $d_k \in \mathbb Z$ the absolute discriminant and by $\mathfrak d_k$ the different ideal. 
For a fractional $\mathcal O_k$-ideal $\mathfrak a \subset k$, we denote by $[\mathfrak a] \in Cl_k$ the ideal class of $\mathfrak a$. 

We equip $\mathbb C$ (resp. $\mathbb C^n$) with the following normalized absolute value $|~~|$ (resp. $||~~||$),
 \begin{align}
&~|~~|: \mathbb C \rightarrow \mathbb R_{\geq 0};~ x+iy \mapsto (x^2+y^2)^{1/2} \quad( x,y \in \mathbb R),\\
&||~~||: \mathbb C^n \rightarrow \mathbb R_{\geq 0};~ (x_1, \dots, x_n) \mapsto (|x_1|^2+\dots +|x_n|^2)^{1/2}.\label{abs val}
\end{align}
For $\sigma \in S_k$, let us choose an embedding $k_{\sigma} \hookrightarrow \mathbb C$. Then the above absolute value induces an absolute value $|~~|$ (resp. $||~||$) on $k_{\sigma}$ (resp. $k_{\sigma}^n$). This is clearly independent of the chosen embedding. 
For $r \geq 1$, we define the pairing $\braket{~,~}: k_{\infty}^{r} \times k_{\infty}^{r} \rightarrow \mathbb R$ by 
\begin{equation}
\braket{(x_1, \dots, x_{r}),(y_1, \dots, y_{r})}:= \sum_{i=1}^{r} Tr_{k/\mathbb Q}(x_iy_i), \quad~(x_i, y_i \in k_{\infty}).
\end{equation}
Here $Tr_{k/\mathbb Q}: k_{\infty} \rightarrow \mathbb R$ is the field trace map naturally extended to $k_{\infty}$.



Suppose now $k'/k$ is an extension of number fields. For $\sigma \in S_k$, we denote by $S_{k', \sigma}$ the set of places of $k'$ above $\sigma$. Then for $\sigma \in S_k$ and $\tau \in S_{k',\sigma}$, we fix an embedding $k_{\sigma} \subset k'_{\tau}$, and denote by $n_{\tau|\sigma}:=[k'_{\tau}: k_{\sigma}]$ the degree of the local filed extension. 

In the following, we basically fix an extension $E/F$ of number fields in our argument. However, when we introduce general notation, we use $k$ or $k'/k$ to represent an arbitrary number field or extension of number fields.





\section{Generalized upper half space over $F$}\label{setting}

In order to generalize closed geodesics on the modular curve $SL_2(\mathbb Z)\backslash \mathfrak h$, first observe that the geodesic $\varpi$ (in Section \ref{intro}) on $\mathfrak h$ connecting conjugate real quadratic irrationals $\alpha$ and $\bar{\alpha}$ is obtained by the left translation of the imaginary axis $I:=i \mathbb R_{>0} \subset \mathfrak h$ by the matrix 
$W = \small \left(
\begin{array}{cc}
 \alpha & \bar{\alpha} \\
 1 & 1
\end{array}
\right) \in GL_2(\mathbb{R})$. See Figure \ref{figure} below.

\begin{figure}[hbtp]
\begin{center}
\vspace{-2cm}
    \includegraphics[clip,height=12cm,angle=90]{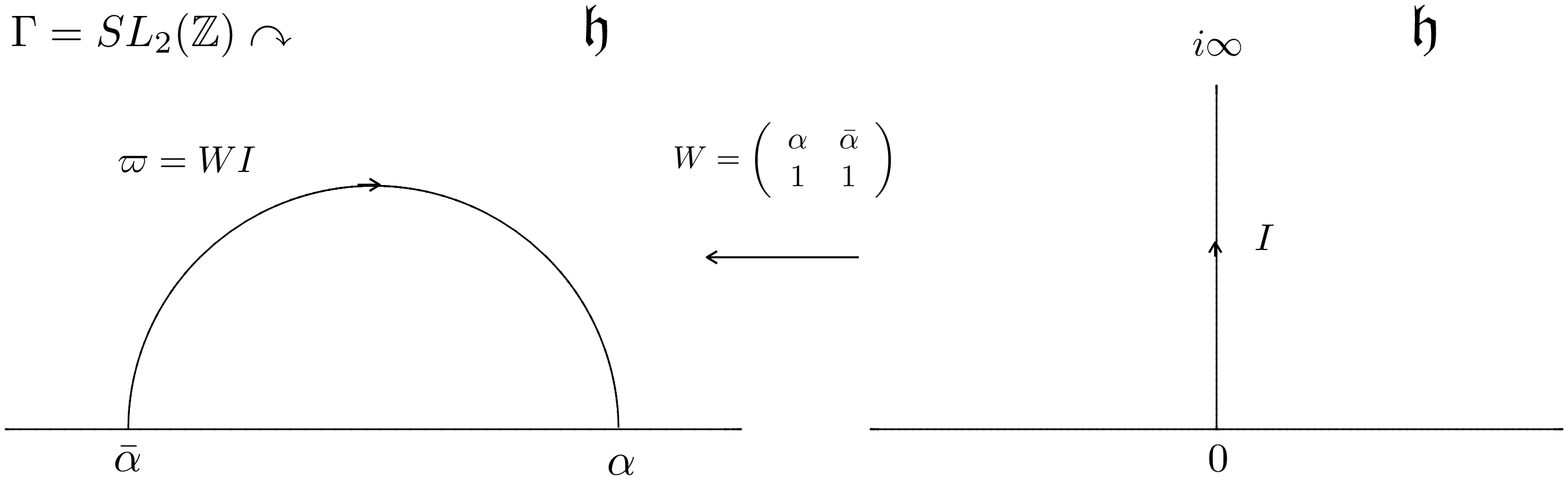} 
    \vspace{-2cm}
    \caption{The geodesic $\varpi$}\label{figure}
\end{center}
\end{figure}
In the following, we generalize each of these objects $\mathfrak h$, $W$, $I$, $\varpi$, and $SL_2(\mathbb Z)$.

Let $F$ be a number field of degree $d$. We fix an embedding $F_{\sigma} \subset \mathbb C$ for each $\sigma \in S_F$. Let us consider an algebraic group $G:= \Res_{F/\mathbb Q}GL_n$, where $\Res_{F/\mathbb Q}$ is the Weil restriction. Then we take a standard maximal compact subgroup $K=\prod_{\sigma \in S_F}K_{\sigma}$ of $G(\mathbb R)\simeq \prod_{\sigma \in S_F} GL_n(F_{\sigma})$ as follows: 
\begin{equation}
K_{\sigma}:=O(n) \subset GL_n(F_{\sigma})=GL_n(\mathbb R), \quad \text{if $\sigma$ is real,}
\end{equation}
\begin{equation}
K_{\sigma}:=U(n) \subset GL_n(F_{\sigma})=GL_n(\mathbb C), \quad \text{if $\sigma$ is complex}.
\end{equation}

\begin{dfn}[Generalized upper half space over $F$, Cf.~\cite{borel77}, \cite{goldfeld06}]
Set 
\begin{equation}
\mathfrak h^n_{\sigma} := GL_n(F_{\sigma})/F_{\sigma}^{\times}K_{\sigma}, \quad \text{for $\sigma \in S_F$.}
\end{equation}
Then we define the generalized upper half space for $G$ to be the symmetric space
\begin{equation}
\mathfrak h^n_F := G(\mathbb R)/F_{\infty}^{\times}K\simeq \prod_{\sigma \in S_F} \mathfrak h^n_{\sigma}.
\end{equation}
In the following, for $g \in G(\mathbb R)=\prod_{\sigma \in S_F}GL_n(F_{\sigma})$, we always denote by $g_{\sigma}$ the $\sigma$-component of $g$, and denote by $[g] \in \mathfrak h^n_F$ the class represented by $g$.
\end{dfn}

\begin{rmk}
In the case where $F=\mathbb Q$, $n=2$, we have an isometry $\mathfrak h^2_{\mathbb Q} \stackrel{\sim}{\rightarrow} \mathfrak h;~ [g] \mapsto g i$, where the action of $g$ on $i \in \mathfrak h$ is the usual linear fractional transformation.
\end{rmk}

\subsection{Heegner objects}\label{heeg obj}
In this section, we define a certain totally geodesic submanifold of $\mathfrak h^n_F$ called the Heegner object. We closely follow the construction in \cite[Section 2]{bekki17}, but have slightly modified the argument in order to deal with the ideal class group of $F$.

Let $F$ be the same as above, and let $E/F$ be a field extension of degree $n$. For $\sigma \in S_F$ and $\tau \in S_{E,\sigma}$, we fix an embedding $E_{\tau} \subset \mathbb C$ so that $F_{\sigma} \subset E_{\tau} \subset \mathbb C$.
Let us fix a basis $w_1, \dots, w_n \in E$ of $E$ over $F$. Set $w:= {}^t\!(w_1 \cdots w_n ) \in E^n$. First, we define two isomorphisms $E_{\infty} \simeq F_{\infty}^n$ of $F_{\infty}$-modules using local and global data on $E_{\infty}$.

\underline{\bf local}: For $\sigma \in S_F$ and $\tau \in S_{E, \sigma}$, we fix an isomorphism $E_{\tau} \simeq F_{\sigma}^{n_{\tau|\sigma}}$ of $F_{\sigma}$-vector spaces as follows: If $n_{\tau|\sigma}=1$, then simply $E_{\tau}=F_{\sigma}$. Otherwise, $\sigma$ is real and $\tau$ is complex, and $n_{\tau|\sigma}=2$. Therefore using $\mathbb C \simeq \mathbb R^2; x+iy \mapsto (y,x)$, we obtain $E_{\tau} = \mathbb C \simeq \mathbb R^2 = F_{\sigma}^2$. This induces
\begin{equation}
E_{\sigma}:= E \otimes_F F_{\sigma} \simeq \prod_{\tau | \sigma}E_{\tau} \simeq \prod_{\tau|\sigma}F_{\sigma}^{n_{\tau|\sigma}} \simeq F_{\sigma}^n,
\end{equation}
and by taking the product over $\sigma \in S_F$, we get
\begin{equation}
\iota: E_{\infty} \stackrel{\sim}{\rightarrow} F_{\infty}^n.
\end{equation}

\underline{\bf global}: Since $w_1,\dots,w_n$ is a basis of $E$ over $F$, we have an isomorphism $w: F^n \stackrel{\sim}{\rightarrow} E; x \mapsto x\cdot w$, of $F$-vector spaces. Here we regard $x$ as a row vector, and $x \cdot w$ is the scalar product. Thus by tensoring $\mathbb R$ over $\mathbb Q$, we obtain
\begin{equation}\label{isom w}
w: F_{\infty}^n \stackrel{\sim}{\rightarrow} E_{\infty}.
\end{equation}
Then we define $W \in G(\mathbb R)$ so that $W=\iota \circ w$ in $\End(F_{\infty}^n)$, that is, 
\begin{equation}
\iota \circ w:  F_{\infty}^n \stackrel{\sim}{\rightarrow}  E_{\infty} \stackrel{\sim}{\rightarrow} F_{\infty}^n;~ x \mapsto xW.
\end{equation}

Next, we generalize $I$. For a number field $k$, we set $T_k:= \prod_{\sigma \in S_k} \mathbb R_{>0}$. Then $T_k$ acts naturally on $k_{\infty} \simeq \prod_{\sigma \in S_k} k_{\sigma}$ by the component-wise multiplication. For an extension $k'/k$ of number fields, the field norm map $N_{k'/k}: k' \rightarrow k$ induces a homomorphism,
\begin{equation}
N_{k'/k}: T_{k'} \rightarrow T_k ;~ (t_{\tau})_{\tau \in S_{k'}} \mapsto \Big(\prod_{\tau|\sigma}t_{\tau}^{n_{\tau|\sigma}}\Big)_{\sigma \in S_k}.
\end{equation}
We denote by $T_{k'/k} := \ker(T_{k'} \stackrel{N_{k'/k}}{\rightarrow}T_k)$ the kernel of this norm homomorphism.

Then the action of $T_{E/F}$ on $F_{\infty}^{n}$ via $E_{\infty} \stackrel{\iota}{\simeq}F_{\infty}^n$, which is clearly as $F_{\infty}$-module, induces a group homomorphism
\begin{equation}
I=I_{E/F}: T_{E/F} \rightarrow G(\mathbb R)=GL_n(F_{\infty}).
\end{equation}

\begin{dfn}[Cf.~\cite{bekki17}]\label{heeg obj dfn}
We define the Heegner object associated to the basis $w$ of $E$ over $F$ by
\begin{equation}
\varpi=\varpi_w : T_{E/F} \rightarrow \mathfrak h_F^n ;~ t \mapsto [WI(t)],
\end{equation}
where $[WI(t)]$ denotes the class of $WI(t)$ in $\mathfrak h_F^n$ as remarked before.
\end{dfn}

\begin{rmk}
If $F=\mathbb Q$ and $E$ is imaginary quadratic, then the image of $\varpi_w$ is just a Heegner point on $\mathfrak h=\mathfrak h_{\mathbb Q}^2$.
\end{rmk}

\paragraph{Arithmetic subgroup and periodicity}
We define an arithmetic subgroup $\Gamma$ of $G(\mathbb R)$ and discuss the periodicity of the Heegner object $\varpi$ with respect to $\Gamma$.
 
Let $L \subset F^n$ be an $\mathcal O_F$-lattice, that is, an $\mathcal O_F$-submodule such that $L\otimes_{\mathcal O_F}F = F^n$. Consider the natural right action of $SL_n(F)$ on the space $F^n$ of row vectors. Define
\begin{equation}\label{gamma}
\Gamma_L:= \Stab_{SL_n(F)}(L)=\{\gamma \in SL_n(F) \mid L\gamma =L \},
\end{equation}
where ``$\Stab$" is the stabilizer subgroup. Then $\Gamma_L$ acts properly discontinuously on $\mathfrak h_F^n$ from the left. 

Now let $\mathfrak A \subset E$ be a fractional $\mathcal O_E$-ideal. We take $L \subset F^n$ so that $L$ corresponds to $\mathfrak A$ under the isomorphism (\ref{isom w}) (i.e., $w : L \stackrel{\sim}{\rightarrow} \mathfrak A$), and set 
\begin{equation}
\Gamma =\Gamma_{\mathfrak A} := \Gamma_L.
\end{equation}

For a number field $k$, we denote by $U_k$ the image of the unit group $\mathcal O_k^{\times}$ under the following ``multiplicative" regulator map:
\begin{equation}\label{abs regulator}
\reg^{\times}_{k}: \mathcal O_{k}^{\times} \rightarrow T_{k} ;~ u \mapsto (|\sigma (u)|)_{\sigma \in S_k}.
\end{equation}
Then, by Dirichlet's unit theorem, $U_k$ is a lattice in $T_{k/\mathbb Q} \subset T_k$, that is, a discrete cocompact subgroup of $T_{k/\mathbb Q}$. 

Now, for an extension $k'/k$ of number fields, let $\mathcal O_{k'/k}^{\times} := \ker (\mathcal O_{k'}^{\times} \stackrel{N_{k'/k}}{\rightarrow} \mathcal O_k^{\times})$ be the relative unit group of $k'/k$. 
We denote by $U_{k'/k}$ the image of $\mathcal O_{k'/k}^{\times}$ under the following relative regulator map $\reg^{\times}_{k'/k}$:
\begin{equation}\label{regulator}
\reg^{\times}_{k'/k}:= \reg^{\times}_{k'}|_{\mathcal O_{k'/k}^{\times}}: \mathcal O_{k'/k}^{\times} \rightarrow T_{k'/k},
\end{equation}
which is just the restriction of $\reg^{\times}_{k'}$.
Then we easily see that $U_{k'/k}$ is a lattice in $T_{k'/k}$.

Let $E/F$ be as before, and let $\pi : \mathfrak h_F^n \rightarrow \Gamma \backslash \mathfrak h_F^n$ be the natural projection.
\begin{prop}\label{periodicity}
The map $T_{E/F} \overset{\varpi}{\rightarrow} \mathfrak h_F^n \overset{\pi}{\rightarrow} \Gamma \backslash \mathfrak h_F^n$ factors through 
\begin{equation}
 \overline{\varpi} :T_{E/F}/U_{E/F} \rightarrow \Gamma \backslash \mathfrak h_F^n.
\end{equation} 
\end{prop}
In order to prove this proposition, we prepare a lemma. Let 
\begin{equation}
\varrho_w: E^{\times} \rightarrow \Aut_F(E) \simeq GL_n(F) \subset G(\mathbb R)
\end{equation}
be the regular representation of $E$ over $F$ with respect to the basis $w$ of $E$ over $F$, that is, for any row vector $x \in F^n$, we have $x\varrho_w(\alpha)w = \alpha xw$ in $E$. 

\begin{lem}\label{lem psi}
For $u \in \mathcal O_{E/F}^{\times}$, we have $W^{-1}\varrho_w(u)W \equiv I(\reg^{\times}_{E/F}(u)) \mod K$.
\end{lem}
\begin{proof}
Set $g := W^{-1}\varrho_w(u)W$ and $\rho := \reg^{\times}_{E/F}(u)$. 
By the definition of $W$ and $\varrho_w$, the left hand side $g=W^{-1}\varrho_w(u)W$ represents the multiplication by $u$ on $E_{\infty}$ via $E_{\infty} \overset{\iota}{\simeq} F_{\infty}^n$. 
Now, for $\tau \in S_E$, the multiplication by $\tau (u)$ on $E_{\tau} \subset \mathbb C$ decomposes into the scaling by $|\tau (u)|$ and a rotation. Therefore $g$ decomposes as $g=I(\rho) R$ for some $R \in K$. 
\end{proof}

\begin{proof}[Proof of Proposition \ref{periodicity}]
 Let $\rho \in U_{E/F}$. We have to show that there exists $\gamma \in \Gamma$ such that $\varpi(\rho t)=\gamma \varpi(t)$ holds for all $t \in T_{E/F}$. Take any $u \in \mathcal O_{E/F}^{\times}$ such that $\reg^{\times}_{E/F}(u)=\rho$. Let $\gamma := \varrho_w(u)$. Since $N_{E/F}(u)=1$ and the multiplication by $u$ preserves the ideal $\mathfrak A$, we have $\gamma \in \Gamma$.
Then, by Lemma \ref{lem psi}, we obtain $\varpi(\rho t)=[WI(\rho t)] = [\gamma WI(t)] =\gamma \varpi(t)$ for all $t\in T_{E/F}$. 
\end{proof}

\subsection{Langlands Eisenstein series for $\Res_{F/\mathbb Q}GL_n$}\label{eis series}
In this section, we set up basic definitions of Eisenstein series on $\mathfrak h_F^n$. 
We apply the general construction of the so-called Langlands Eisenstein series to our case $G=\Res_{F/\mathbb Q}GL_n$. We basically follow the argument in Goldfeld~\cite{goldfeld06}. However, since \cite{goldfeld06} deals only with the case  where $F= \mathbb Q$, we need some additional consideration.

In order to define the Langalnds Eisenstein series, we have to choose a parabolic subgroup of $\Gamma_L$.
Let us denote by $\data{\mathfrak a}{L}$ a data consisting of
\begin{enumerate}[--]
\item $L \subset F^n$, an $\mathcal O_F$-latiice (not necessarily defined from the fractional ideal $\mathfrak A \subset E$),
\item $\mathfrak a \subset F$, a fractional $\mathcal O_F$-ideal such that $1 \in \mathfrak a$ (we call such $\mathfrak a$ {\it an anti-integral ideal}), 
\item $\data{\mathfrak a}{L}$, a split injective $\mathcal O_F$-homomorphism, that is, the exact sequence $0 \rightarrow \mathfrak a \hookrightarrow L \rightarrow L/\mathfrak ae \rightarrow 0$ splits, where we denote by $e \in L$ the image of $1 \in  \mathfrak a$. 
\end{enumerate}
In the following, we refer to such a data $\data{\mathfrak a}{L}$ {\it a parabolic data}. 
Set $\Gamma := \Gamma_L$ (cf. (\ref{gamma})). Then we define the parabolic subgroup $P_{\data{\mathfrak a}{L}}$ associated to the data $\data{\mathfrak a}{L}; 1 \mapsto e$ as 
\begin{equation}\label{parabolic}
P= P_{\data{\mathfrak a}{L}} := \Stab_{\Gamma_L}(\data{\mathfrak a}{L})=\{\gamma \in \Gamma_L \mid \mathfrak ae\gamma =\mathfrak ae ~~\text{in $L$}\}.  
\end{equation}

Next, we define a certain left $P$-invariant function on $\mathfrak h_F^n$. 
\begin{dfn}
We define a function 
$
\Det=\Det_{\data{\mathfrak a}{L}}: \mathfrak h_F^n \rightarrow \mathbb R_{\geq 0} $ 
 by
\begin{equation}
\Det([g]) := \prod_{\sigma \in S_F} \frac{|\det g_{\sigma}|^{n_{\sigma}}}{||e g_{\sigma}||^{nn_{\sigma}}}, \quad\text{for $g=(g_{\sigma})_{\sigma} \in G(\mathbb R)$},
\end{equation}
where $||~~||$ is the normalized absolute value defined as (\ref{abs val}). 
Here by an abuse of notation, we denote also by $e$ the image of $e \in L \subset F^n$ under the completion map $\sigma: F^n \hookrightarrow F_{\sigma}^n$, and $eg_{\sigma}$ is an element of $F_{\sigma}^n$.  It is clear that this function is well-defined on $\mathfrak h_F^n$. 
\end{dfn}

\begin{lem}
The function $\Det(z)$ is left $P$-invariant, that is, $\Det(\gamma z) =\Det(z)$ for all $z \in \mathfrak h_F^n$ and $\gamma \in P$.
\end{lem}

\begin{proof}
Since $\gamma \in \Gamma_L$, we have $\det \gamma=1$. Now, since $\gamma \in P$ preserves $\mathfrak ae \subset L$, $\gamma$ acts on $e$ by the multiplication by some $u \in \mathcal O_F^{\times}$. Then,  $\prod_{\sigma}||e\gamma g_{\sigma}||^{n_{\sigma}}=\prod_{\sigma}||ueg_{\sigma}||^{n_{\sigma}}=\prod_{\sigma}|\sigma(u)|^{n_{\sigma}}||eg_{\sigma}||^{n_{\sigma}} =\prod_{\sigma}||eg_{\sigma}||^{n_{\sigma}} $. This shows the lemma.
\end{proof}

\begin{dfn}\label{eis ser dfn}
We define the Eisenstein series associated to the data $\data{\mathfrak a}{L}$ as
\begin{equation}\label{eis ser eqn}
E_{\data{\mathfrak a}{L}}(z,s) := \sum_{\gamma \in P \backslash \Gamma} \Det(\gamma z)^s, \quad\text{for $z\in \mathfrak h_F^n$ and $s\in \mathbb C$, $\re s >1$}.
\end{equation}
\end{dfn}
We prove the absolute convergence of $E_{\data{\mathfrak a}{L}}(z,s)$ in Section \ref{conv eis ser}.

\paragraph{More explicit form}
In order to study $E_{\data{\mathfrak a}{L}}(z,s)$, we rewrite the sum more explicitly.
For an $\mathcal O_F$-lattice $L \subset F^n$ and a fractional $\mathcal O_F$-ideal $\mathfrak a \subset F$, we define
\begin{align}
L_{\mathfrak a}&:= \{ x \in L-\{0\} \mid Fx \cap L = \mathfrak ax \}\\
			&~= \{x \in L \mid \mathfrak a \rightarrow L;~ \alpha \mapsto \alpha x, ~\text{is a split injective $\mathcal O_F$-homomorphism} \}. \label{l_b}
\end{align}

\begin{lem}\label{a-compo lem}
\begin{enumerate}[{\rm (1)}]
\item We have a decomposition $L -\{0\} = \coprod\limits_{\mathfrak a:\text{\rm anti-int.}} L_{\mathfrak a}$, where the union is taken over the anti-integral ideals.
\item For a parabolic data $\data{\mathfrak a}{L}; 1\mapsto e$, we have the following bijection:
\begin{equation}
P_{\data{\mathfrak a}{L}}\backslash \Gamma_L \overset{\sim}{\longrightarrow} \mathcal O_F^{\times} \backslash L_{\mathfrak a};~  \gamma \mapsto e \gamma.
\end{equation}
\end{enumerate}
\end{lem}

We need the following structure theorem for finitely generated projective modules over Dedekind domains. See \cite[Chapter 7, \S 4, Proposition 24]{bourbakicommalg} for example.
\begin{prop}\label{str thm}
Let $A$ be a Dedekind domain. Then any non-zero finitely generated projective module $M$ over $A$ is isomorphic to $A^{r-1}\oplus \mathfrak a$ for some $r \in \mathbb Z_{\geq 1}$ and a fractional $A$-ideal $\mathfrak a$. Moreover $r$ and the ideal class of $\mathfrak a$ in this presentation are unique. \qed
\end{prop}

\begin{proof}[Proof of Lemma \ref{a-compo lem}]
(1) is clear. To see (2), it suffices to show that the right (matrix) action of $\Gamma_L$ on $\mathcal O_F^{\times}\backslash L_{\mathfrak a}$ is transitive. 
Take any $x \in L_{\mathfrak a}$. Then by (\ref{l_b}), there exist isomorphisms 
$\varphi_1: \mathfrak a \oplus L/\mathfrak ae \overset {\sim}{\rightarrow} L$ which sends $(1,0)$ to $e$, and  $\varphi_2: \mathfrak a \oplus L/\mathfrak ax \overset {\sim}{\rightarrow} L$ which sends $(1,0)$ to $x$. Then by Proposition \ref{str thm}, we see that there exists an $\mathcal O_F$-isomorphism $\varphi: L/\mathfrak ae \overset{\sim}{\rightarrow}  L/\mathfrak ax$.
Then for any $u \in \mathcal O_F^{\times}$, we obtain an automorphism 
\begin{equation}
L \overset{\varphi_1^{-1}}{\overset {\sim}{\longrightarrow}} \mathfrak a \oplus L/\mathfrak ae \overset{\times u \oplus \varphi}{\overset {\sim}{\longrightarrow}} \mathfrak a \oplus L/\mathfrak ax \overset{\varphi_2}{\overset {\sim}{\longrightarrow}} L
\end{equation}
of $L$ which sends $e$ to $x$. This extends to $\gamma \in GL_n(F)$ such that $\det \gamma \in \mathcal O_F^{\times}$. By replacing $u$ with $u (\det \gamma)^{-1}$, we obtain $\gamma \in \Gamma_L$ such that $e \gamma \equiv x \mod \mathcal O_F^{\times}$. 
\end{proof}

\begin{cor}\label{eis ser lem}
We can rewrite the sum (\ref{eis ser eqn}) as
\begin{equation}\label{eis ser eqn2}
E_{\data{\mathfrak a}{L}}([g],s) = \sum_{x \in \mathcal O_F^{\times} \bs L_{\mathfrak a}} \prod_{\sigma \in S_F}\frac{|\det g_{\sigma}|^{n_{\sigma}s}}{||x g_{\sigma}||^{nn_{\sigma}s}}, \quad \text{for $g\in G(\mathbb R)$}.
\end{equation}
In particular, $E_{\data{\mathfrak a}{L}}(z,s)$ depends only on the lattice $L$ and the anti-integral ideal $\mathfrak a$.
\end{cor}

We define some variants of $E_{\data{\mathfrak a}{L}}(z,s)$. 
\begin{dfn}\label{var eis ser}
Using the expression (\ref{eis ser eqn2}), for $\mathscr A \in Cl_F$, we define
\begin{equation}
E_{L, \mathscr A}(z,s) := \sum_{\substack{\mathfrak a \in \mathscr A\\ \mathfrak a:\text{anti-int.}}}E_{\data{\mathfrak a}{L}}(z,s),
\end{equation}
where the sum is taken over the anti-integral ideals in $\mathscr A$, and define
\begin{equation}
E_{L}(z,s) := \sum_{\mathscr A\in Cl_F}E_{L, \mathscr A}(z,s) = \sum_{x \in \mathcal O_F^{\times} \bs L-\{0\}} \prod_{\sigma \in S_F}\frac{|\det g_{\sigma}|^{n_{\sigma}s}}{||x g_{\sigma}||^{nn_{\sigma}s}}.
\end{equation}
Here the last equality follows from Lemma \ref{a-compo lem} (1) and Corollary \ref{eis ser lem}.
\end{dfn}

\begin{ex}\label{eis ser ex}
In the case where $F=\mathbb Q$, $n=2$, $L=\mathbb Z \oplus \mathbb Z$, and $\mathfrak a=\mathbb Z$, we have $L_{\mathfrak a}=\{(c,d) \in \mathbb Z^2 \mid (c,d)=1 \}$, and $E_{\data{\mathbb Z}{\mathbb Z^2}}(z,s)$ is nothing but the usual real analytic Eisenstein series (\ref{class eis ser}).
\end{ex}

\section{Hecke's integral formula}\label{integral formula}

In this section, we first show the convergence of Eisenstein series defined in Section \ref{eis series}. Then we prove the relative Hecke's integral formula for $E/F$. 

\subsection{Haar measures}\label{haar meas}
First, we fix the normalization of the Haar measures. 
Let $k$ be an arbitrary number field. For $\sigma \in S_k$, we normalize the Haar measure $dx_{\sigma}$ (resp. $\dtimes t_{\sigma}$) on $k_{\sigma}$ (resp. $\mathbb R_{>0}$) as
\begin{align}
dx_{\sigma} &= dx, &\dtimes t_{\sigma}&= dt/t,          & &\text{if $\sigma$ is real}, \\
dx_{\sigma} &= i dxd\bar{x}, &\dtimes t_{\sigma}&= dt^{2}/t^{2},                         & &\text{if $\sigma$ is complex}.
\end{align}
Then we define the Haar measure $dx_k$ (resp. $\dtimes t_k$) on $k_{\infty}$ (resp. $T_k$) to be the product measure $ dx_k:= \prod_{\sigma \in S_k}dx_{\sigma}$ (resp. $\dtimes t_k:= \prod_{\sigma \in S_k}\dtimes t_{\sigma}$).

Now, let $k'/k$ be an arbitrary extension of number fields. We normalize the Haar measure $\dtimes t_{k'/k}$ on $T_{k'/k}$ so that the induced Haar measure on the quotient group $T_k \simeq T_{k'}/T_{k'/k}$ coincides with $\dtimes t_k$, that is, for any integrable function $\phi$ on $T_{k'}$, we have
\begin{equation}
\int_{T_{k'}} \phi ~\dtimes t_{k'} = \int_{T_k} \Big( \int_{T_{k'/k}} \phi ~\dtimes t_{k'/k}  \Big) \dtimes t_k.
\end{equation}
Next, we fix the notion of the relative gamma factor for $k'/k$. 
\begin{lem}\label{lem dtimes}
Let $r$ be a positive integer, and let also $n_1, \dots, n_r$ be positive integers. Set $T:= \{t=(t_i)_i \in \mathbb R_{>0}^r \mid \prod_{i=1}^r t^{n_i} =1 \}$. Take any $i_0 \in \{1, \dots , r\}$, and let
\begin{equation}
p^{i_0}: \mathbb R_{>0}^r \rightarrow \mathbb R_{>0}^{r-1} :~ (t_i)_{1\leq i\leq r} \mapsto (t_i)_{1\leq i\leq r, i \neq i_0}
\end{equation}
be the natural projection, which induces an isomorphism $p^{i_0}: T \overset{\sim}{\rightarrow} \mathbb R_{>0}^{r-1}$. We define the Haar measure $\dtimes t$ on $T$ to be the pull-back measure $\dtimes t :=(p^{i_0})^*\prod_{i\neq i_0} dt^{n_i}/t^{n_i}$. 
\begin{enumerate}[{\rm (1)}]
\item For any integrable function $\phi$ on $\mathbb R_{>0}^r$, the measure $\dtimes t$ satisfies
\begin{equation}
\int_{\mathbb R_{>0}^r} \phi(t) ~\prod_{i=1}^r\frac{dt_i^{n_i}}{t_i^{n_i}} = \int_{u=t_1^{n_1}\cdots t_r^{n_r} \in \mathbb R_{>0}} \Big( \int_{T} \phi(t) \,\dtimes t  \Big) \frac{du}{u}.
\end{equation}
In particular, the Haar measure $\dtimes t$ is independent of the choice of $i_0$.
\item Set $n:=\sum_{i=1}^r n_i$, and let $\Gamma(s)$ be the gamma function. Then
\begin{equation}
n \Gamma(ns) \int_{T} \frac{1}{(t_1+\cdots +t_r)^{ns}} \dtimes t= \prod_{i=1}^r n_i\Gamma(n_i s).
\end{equation}
\end{enumerate}
\end{lem}
\begin{proof}
This can be proved by a straightforward computation, and we omit the proof.
\end{proof}
Let $k'/k$ be an arbitrary extension of number fields of degree $n$. We define the relative gamma factor $\Gamma_{k'/k}(s)$ $(s \in \mathbb C, \re s >1)$ as
\begin{align}
\Gamma_{k'/k}(s):= \int_{T_{k'/k}}\prod_{\sigma \in S_k}\frac{1}{(\sum_{\tau \in S_{k', \sigma}} t_{\tau}^2)^{nn_{\sigma}s/2}}\dtimes t_{k'/k}. \label{gamma factor}
\end{align}
\begin{lem}\label{lem gamma factor}
We have
\begin{equation}
\Gamma_{k'/k}(s)= \frac{\prod_{\tau \in S_{k'}}\frac{n_{\tau}}{2}\Gamma(\frac{n_{\tau}s}{2})}{\prod_{\sigma \in S_k}\frac{nn_{\sigma}}{2}\Gamma(\frac{nn_{\sigma}s}{2})}, \quad\text{for $\re s>1$}.
\end{equation}
\end{lem}
\begin{proof}
This follows from Lemma \ref{lem dtimes}.
\end{proof}

Let again $k$ (resp. $k'/k$) be a number field (resp. an extension of number fields). We define the regulator $R_{k}$ (resp. relative regulator $R_{k'/k}$) to be the volume of $T_{k/\mathbb Q}/U_{k}$ (resp. $T_{k'/k}/U_{k'/k}$) with respect to the Haar measure $\dtimes t_{k/\mathbb Q}$ (resp. $\dtimes t_{k'/k}$).
We define $\mu_k$ (resp. $\mu_{k'/k}$) to be the subgroup of torsion elements in $\mathcal O_{k}^{\times}$ (resp. $\mathcal O_{k'/k}^{\times}$), and set $w_{k}:=\# \mu_{k}$ (resp. $w_{k'/k}:=\# \mu_{k'/k}$). 
\begin{lem}\label{lem regulator}
Let $k'/k$ be an extension of number fileds. Then we have 
\begin{equation}
R_{k'/k} = \frac{[U_{k'}: U_{k}U_{k'/k}]}{[k':k]^{r_k-1}} \frac{R_{k'}}{R_{k}}.
\end{equation}
\end{lem}
\begin{proof}
This follows easily from the exact sequence $1 \rightarrow U_{k'/k} \rightarrow U_{k}U_{k'/k} \rightarrow N_{k'/k}U_k \rightarrow 1$. See also \cite[Theorem 1]{costafriedman91}, although their definition of the relative regulator is slightly different from ours.
\end{proof}

\subsection{The convergence of Eisenstein series}\label{conv eis ser}
Now we prove the convergence of Eisenstein series defined in Section \ref{eis series}. 
Let $F$ be the same as in Section \ref{setting}. Note that we can identify $T_{F/\mathbb Q}$ as a subgroup of $G(\mathbb R)=GL_n(F_{\infty})$ lying in the center $F_{\infty}^{\times} \subset GL_n(F_{\infty})$.
\begin{lem} 
The infinite series
\begin{equation}
\mathcal E([g],s):= \sum_{x \in L-\{0\}} \frac{1}{(\sum_{\sigma \in S_F} ||x g_{\sigma}||^{2})^{dns/2}},
\end{equation}
converges absolutely and compactly for $\re s >1$ and $g=(g_{\sigma})_{\sigma} \in G(\mathbb R)$. 
Moreover, for any $g \in G(\mathbb R)$ and $\rho \in U_{F}$, we have
$\mathcal E_L([g\rho],s)=\mathcal E_L([g],s)$. 
\end{lem}
\begin{proof}
This is the well-known convergence of the Epstein zeta function. See \cite[p.47]{siegel80} for example.
\end{proof}

\begin{prop}\label{convergence e_l}
We have
\begin{equation}
\int_{T_{F/\mathbb Q} / U_{F}} \mathcal E([g t],s)\dtimes t_{F/\mathbb Q} =\frac{w_{F}\Gamma_{F/\mathbb Q}(ns)}{\prod_{\sigma \in S_F}|\det g_{\sigma}|^{n_{\sigma}s}}E_L(z,s),
\end{equation}
for $\re s >1$ and $g=(g_{\sigma})_{\sigma} \in G(\mathbb R)$. In particular, $E_{L}(z,s)$ converges absolutely and compactly for $\re s >1$ and $g=(g_{\sigma})_{\sigma} \in G(\mathbb R)$.
\end{prop}
\begin{proof}We see this by the classical argument as follows. 
\begin{align}
LHS &=\int_{T_{F/\mathbb Q} / U_{F}} \sum_{x \in L-\{0\}} \frac{1}{(\sum_{\sigma \in S_F} ||x g_{\sigma}||^{2}t_{\sigma}^2)^{dns/2}}\dtimes t_{F/\mathbb Q}\\
&= w_{F} \int_{T_{F/\mathbb Q}} \sum_{x \in \mathcal O_{F}^{\times}\bs L-\{0\}} \frac{1}{(\sum_{\sigma \in S_F} ||x g_{\sigma}||^{2}t_{\sigma}^2)^{dns/2}}\dtimes t_{F/\mathbb Q}\\
&=w_{F} \sum_{x \in \mathcal O_{F}^{\times}\bs L-\{0\}} \int_{T_{F/\mathbb Q}} \frac{1}{(\sum_{\sigma \in S_F} ||x g_{\sigma}||^{2}t_{\sigma}^2)^{dns/2}}\dtimes t_{F/\mathbb Q}. \label{conv eqn1}
\end{align}
Now, since $\rho :=\Bigg(\frac{||xg_{\sigma}||}{\prod_{\sigma' \in S_F}||xg_{\sigma'}||^{n_{\sigma'}/d}}\Bigg)_{\sigma} \in T_{F/\mathbb Q}$, the Haar measure $\dtimes t_{F/\mathbb Q}$ is invariant under the change of variables $t \mapsto \rho^{-1} t$.
Then we have
\begin{align}
LHS 
=w_{F} \sum_{x \in \mathcal O_{F}^{\times}\bs L-\{0\}} \frac{1}{\prod_{\sigma' \in S_F}||xg_{\sigma'}||^{nn_{\sigma'}s}}\int_{T_{F/\mathbb Q}} \frac{1}{(\sum_{\sigma \in S_F} t_{\sigma}^2)^{dns/2}}\dtimes t_{F/\mathbb Q}.
\end{align}
The proposition now follows from (\ref{gamma factor}).
\end{proof}

\begin{prop}\label{prop eis ser}
Let $\data{\mathfrak a}{L}$ be a parabolic data, and let $\mathscr A\in Cl_F$.
\begin{enumerate}[{\rm (1)}]
\item  Both $E_{\data{\mathfrak a}{L}}(z,s)$ (Definition \ref{eis ser dfn}, Corollary \ref{eis ser lem}) and $E_{L,\mathscr A}(z,s)$ (Definition \ref{var eis ser}) converge absolutely and compactly for $\re s>1$ and $z \in \mathfrak h_F^n$. 
\item The Eisenstein series $E_{\data{\mathfrak a}{L}}(z,s)$ is an automorphic function on $\mathfrak h_F^n$ with respect to $\Gamma_L$, that is, we have
  \begin{equation}
    E_{\data{\mathfrak a}{L}}(\gamma z,s)=E_{\data{\mathfrak a}{L}}(z,s), \quad\text{for all $\gamma \in \Gamma_L$}.
  \end{equation}
\item We have
  \begin{equation}\label{eis ser eqn3}
    E_{L,[\mathfrak a]}(z,s)=\frac{\zeta_F(\mathfrak a^{-1},ns)}{(N\mathfrak a)^{ns}}E_{\data{\mathfrak a}{L}}(z,s).
  \end{equation}
\end{enumerate}
\end{prop}
\begin{proof}
(1) follows directly from Proposition \ref{convergence e_l}, and (2) follows directly from the definition of $E_{\data{\mathfrak a}{L}}(z,s)$ (Definition \ref{eis ser dfn}).
To prove (3), it suffices to see
  \begin{equation}\label{eis ser eqn4}
    E_{\data{\alpha \mathfrak a}{L}}(z,s)=|N_{F/\mathbb Q}(\alpha)|^{ns}E_{\data{\mathfrak a}{L}}(z,s), \quad\text{for $\alpha \in F^{\times}$}.
  \end{equation}
This follows from the identity $L_{\alpha \mathfrak a}=\frac{1}{\alpha}L_{\mathfrak a}$ and Corollary \ref{eis ser lem}.
\end{proof}

\subsection{Relative Hecke's integral formula for $E/F$}\label{rel int formula}
Now we prove our first main theorem. Let the notations be the same as in Section \ref{heeg obj}. That is, $E/F$ is an extension of number fields of degree $n$, $\varpi$ is the Heegner object associated to a basis $w$ of $E$ over $F$, $L \subset F^n$ is the $\mathcal O_F$-lattice corresponding to a fractional $\mathcal O_E$-ideal $\mathfrak A$ with respect to the basis $w$, and $\Gamma=\Gamma_{L}$ is the arithmetic subgroup of $SL_n(F)$ associated to $L$. Let $\mathfrak a$ be an anti-integral $\mathcal O_F$-ideal and let $\data{\mathfrak a}{L}$ be a parabolic data. We consider the Eisenstein series $E_{\data{\mathfrak a}{L}}(z,s)$ associated to this data.
In this section, we assume $\re s>1$.

\begin{thm}\label{crude int formula}
Put $\Delta_w := N_{F/\mathbb Q}(d_w)$, where $d_w:= (\det W)^2 \in F$ is the ``discriminant" of the basis $w$. 
Then we have 
  \begin{align}
    &\int_{T_{E/F}/U_{E/F}}E_{\data{\mathfrak a}{L}}(\overline{\varpi}(t),s)\dtimes t_{E/F} \nonumber \\
    &=|\Delta_w|^{\frac{s}{2}}R_{E/F}\frac{w_E R_E^{-1} \Gamma_{E/\mathbb Q}(s)}{w_F R_F^{-1}\Gamma_{F/\mathbb Q}(ns)}\sum_{x\in \mathcal O_E^{\times} \bs \mathfrak A_{\mathfrak a}}\frac{1}{|N_{E/\mathbb Q}(x)|^s},
  \end{align}
for $\re s >1$, where $\mathfrak A_{\mathfrak a}$ is the image of $L_{\mathfrak a}$ under the isomorphism $w$ (\ref{isom w}), or equivalently, the subset of $\mathfrak A$ defined in the same way as (\ref{l_b}).
\end{thm}

\begin{proof}
 The proof is similar to that of Proposition \ref{convergence e_l}. Let $W=(W_{\sigma})_{\sigma} \in G(\mathbb R)$, and $I(t)=(I(t)_{\sigma})_{\sigma} \in G(\mathbb R)$ ($t \in T_{E/F}$) be as in Section \ref{heeg obj} so that $\varpi(t)=[WI(t)]$. Let us fix a maximal  torsion free subgroup $\tilde{U}_E$ (resp. $\tilde{U}_F,~ \tilde{U}_{E/F}$) of $\mathcal O_E^{\times}$ (resp. $\mathcal O_F^{\times},~  \mathcal O_{E/F}^{\times}$) so that $\tilde{U}_{F}, \tilde{U}_{E/F} \subset \tilde{U}_{E}$. Note that we have $\reg^{\times}_{E/F}: \tilde{U}_{E/F} \overset{\sim}{\rightarrow} U_{E/F}$, etc.
Using Corollary \ref{eis ser lem}, we obtain
 \begin{align}
   LHS 
   &= \frac{1}{w_F} \int_{T_{E/F}/U_{E/F}} \sum_{x \in \tilde U_F \bs L_{\mathfrak a}} \prod_{\sigma}\frac{|\det W_{\sigma}|^{n_{\sigma}s}}{||xW_{\sigma}I(t)_{\sigma}||^{nn_{\sigma}s}} \dtimes t_{E/F}\\
   &= \frac{1}{w_F} \int_{T_{E/F}/U_{E/F}} \sum_{u\in \tilde U_{E/F}}\sum_{x \in \tilde U_F \tilde U_{E/F} \bs L_{\mathfrak a}} \prod_{\sigma}\frac{|\det W_{\sigma}|^{n_{\sigma}s}}{||x\varrho_w(u)W_{\sigma}I(t)_{\sigma}||^{nn_{\sigma}s}} \dtimes t_{E/F}\\
   &= \frac{1}{w_F} \int_{T_{E/F}/U_{E/F}} \sum_{u\in U_{E/F}}\sum_{x \in \tilde U_F \tilde U_{E/F} \bs L_{\mathfrak a}} \prod_{\sigma}\frac{|\det W_{\sigma}|^{n_{\sigma}s}}{||xW_{\sigma}I(ut)_{\sigma}||^{nn_{\sigma}s}} \dtimes t_{E/F} \label{use lem psi}\\
   &= \frac{1}{w_F} \int_{T_{E/F}}\sum_{x \in \tilde U_F \tilde U_{E/F} \bs L_{\mathfrak a}} \prod_{\sigma}\frac{|\det W_{\sigma}|^{n_{\sigma}s}}{||xW_{\sigma}I(t)_{\sigma}||^{nn_{\sigma}s}} \dtimes t_{E/F} \label{eq 5}
 \end{align}
 Here, the action of $\tilde U_{E/F}$ on $L_{\mathfrak a}$ is defined by the regular representation $\varrho_w$, and the equality (\ref{use lem psi}) follows from Lemma \ref{lem psi}. By putting $z=xw$, we obtain
  \begin{align}
LHS &= \frac{|\Delta_w|^{\frac{s}{2}}}{w_F} \sum_{z \in \tilde U_F \tilde U_{E/F} \bs \mathfrak A_{\mathfrak a}}\int_{T_{E/F}} \prod_{\sigma}\frac{1}{(\sum_{\tau \in S_{E,\sigma}}|\tau (z)|^2t_{\tau}^2)^{nn_{\sigma}s/2}} \dtimes t_{E/F}\label{eq 6}
 \end{align}
 Now, since $\rho :=\Big(|\tau(z)|\prod_{\tau' \in S_{E,\sigma_{\tau}}}|\tau'(z)|^{-n_{\tau'}/nn_{\sigma_{\tau}}}\Big)_{\tau} \in T_{E/F}$, where $\sigma_{\tau}$ is the place of $F$ below $\tau$, the Haar measure $\dtimes t_{E/F}$ is invariant under the change of variables $t \mapsto \rho^{-1} t$. Therefore, we have
\begin{align}
LHS &= \frac{|\Delta_w|^{\frac{s}{2}}}{w_F} \sum_{z \in \tilde U_F \tilde U_{E/F} \bs \mathfrak A_{\mathfrak a}} \frac{1}{|N_{E/\mathbb Q}(z)|^{s}} \int_{T_{E/F}}\prod_{\sigma}\frac{1}{(\sum_{\tau \in S_{E,\sigma}}t_{\tau}^2)^{nn_{\sigma}s/2}} \dtimes t_{E/F}\\
   &=|\Delta_w|^{\frac{s}{2}}[U_E: U_F U_{E/F}]\Gamma_{E/F}(s) \frac{w_E}{w_F} \sum_{z \in \mathcal O_{E}^{\times} \bs \mathfrak A_{\mathfrak a}} \frac{1}{|N_{E/\mathbb Q}(z)|^{s}}
\end{align} 
Now the theorem follows from Lemma \ref{lem gamma factor} and Lemma \ref{lem regulator}.
\end{proof}

\begin{dfn}\label{relative partial zeta}
Let $\mathfrak A$ be the same as above. For $\mathscr A \in Cl_F$, we define the partial zeta function $\zeta_{E/F,\mathscr A}(\mathfrak A^{-1},s)$ associated to $\mathfrak A^{-1}$ relative to $\mathscr A$ as
\begin{equation}
\zeta_{E/F,\mathscr A}(\mathfrak A^{-1},s):=N(\mathfrak A)^s \sum_{\substack{\mathfrak a \in \mathscr A\\ \text{anti-int.}}}  \sum_{x \in \mathcal O_E^{\times} \bs \mathfrak A_{\mathfrak a}} \frac{1}{|N_{E/\mathbb Q}(x)|^s},
\end{equation}
where $N(\mathfrak A)$ is the absolute norm of the fractional $\mathcal O_E$-ideal $\mathfrak A$. 
\end{dfn}

Now, using Proposition \ref{prop eis ser} (3), we can rewrite Theorem \ref{crude int formula} as follows.
\begin{thm}[Relative Hecke's integral formula]\label{hecke's int formula}~

Let $[\mathfrak a] \in Cl_F$ be the ideal class of $\mathfrak a$. Put 
\begin{equation}
c_{E/F}(s) := R_{E/F}\dfrac{w_E R_E^{-1} \Gamma_{E/\mathbb Q}(s)}{w_F R_F^{-1}\Gamma_{F/\mathbb Q}(ns)}.
\end{equation}
Then we obtain
\begin{align}
  \int_{T_{E/F}/U_{E/F}} E_{\data{\mathfrak a}{L}}(\overline{\varpi}(t),s)\dtimes t_{E/F} 
    =|\Delta_w|^{\frac{s}{2}}c_{E/F}(s)\frac{N(\mathfrak A)^{-s}}{N(\mathfrak a)^{-ns} } \frac{ \zeta_{E/F,[\mathfrak a]}(\mathfrak A^{-1},s)}{ \zeta_F(\mathfrak a^{-1},ns)}.
  \end{align}
\end{thm}

\begin{cor}\label{cor int formula}
Taking the sum over the anti-integral ideals $\mathfrak a$, we obtain
\begin{align} 
\int_{T_{E/F}/U_{E/F}} E_{L}(\overline{\varpi}(t),s)\dtimes t_{E/F} =|\Delta_w|^{\frac{s}{2}}c_{E/F}(s)N(\mathfrak A)^{-s} \zeta_{E}(\mathfrak A^{-1},s).
\end{align}
\end{cor}

\begin{rmk}
In the case where $F=\mathbb Q$ or $n=2$, Corollary \ref{cor int formula} gives the results of Hecke~\cite[p.~370]{hecke17} and Yamamoto~\cite[Theorem 3.1.2]{yamamoto08} respectively. It should be also remarked that in the case where $F$ is imaginary quadratic and $n=2$, or $F=\mathbb Q$ and $E$ is totally real, Harder~\cite{harder82} and Sczech~\cite{sczech93} respectively obtain results which can be seen as a ``cohomological interpretation'' of this theorem, and deduce the rationality of the special values of zeta functions. 
\end{rmk}



\section{The Fourier expansion of Eisenstein series}\label{fourier exp}
In this preliminary section, we present a kind of the Fourier expansion formula of the Eisenstein series $E_{L,[\mathfrak a]}(z,s)$ for $\re s>1$. This shows the analytic continuation of the Eisenstein series, and enables us to compute its residue and the constant term at $s=1$. As remarked in Section \ref{intro}, the Fourier coefficients (Theorem \ref{fourier coeff thm}) may be obtained  from the more general theory. However here we present an explicit formula in terms of ideal classes of $F$, and give a proof in a self-contained way.

\subsection{Setting}
Let $F/\mathbb Q$ be a number field of degree $d$, and let $\data{\mathfrak a}{L}$ be as before. For another $\mathcal O_F$-lattice $L' \subset F^n$ which is isomorphic to $L$, let $\gamma \in GL_n(F)$ be a matrix such that $L\gamma=L'$. Let us consider the parabolic data $\data{\mathfrak a}{L} \overset{\gamma}{\rightarrow} L'$. Then we easily see that
\begin{equation}
E_{\data{\mathfrak a}{L}}(z,s)=|N_{F/\mathbb Q}(\det \gamma)|^s E_{\data{\mathfrak a}{L'}}(\gamma^{-1}z,s).
\end{equation}
On the other hand, by Proposition \ref{str thm} any $\mathcal O_F$-lattice $L \subset F^n$ is isomorphic to the lattice of the form $\mathfrak{a}_1 \oplus \cdots \oplus \mathfrak{a}_n \subset F^n$ for some anti-integral ideals $\mathfrak{a}_i \subset F$ ($i=1, \dots, n$).
Therefore, in this section we assume $L=\mathfrak{a}_1 \oplus \cdots \oplus \mathfrak{a}_n \subset F^n$. We also assume $n \geq 2$.


Let $\mathfrak a_n \overset{i_n}{\hookrightarrow} L$ be the $n$-th inclusion, which is another parabolic data.
We consider the Fourier expansion of the Eisenstein series $E_{L, [\mathfrak a]}(z,s)$ at the ``cusp" corresponding to the parabolic subgroup $P_{\data{\mathfrak a_n}{L}}=\Stab_{\Gamma_L}(\data{\mathfrak a_n}{L})$ associated to the data $\data{\mathfrak a_n}{L}$.
Let 
\begin{align}\label{def of N}
\footnotesize{
N:=\left\{ \left(  \begin{matrix} 
     ~&~& ~ & \xi_1 \\
     ~&I&~ & \vdots \\
     ~&~& ~ & \xi_{n-1} \\
     0&\cdots& 0 & 1 \\      
   \end{matrix}\right)~\middle|~ \xi_{i} \in \mathfrak a_i^{-1}\mathfrak a_n \right\} \subset
N(\mathbb R):=\left\{ \left(  \begin{matrix} 
     ~&~& ~ & \xi_1 \\
     ~&I&~ & \vdots \\
     ~&~& ~ & \xi_{n-1} \\
     0&\cdots& 0 & 1 \\      
   \end{matrix}\right)~\middle|~ \xi_{i} \in F_{\infty} \right\}  }
\end{align}
be the nilpotent radical of $P_{\data{\mathfrak a_n}{L}}$, and its canonical lifting into $G(\mathbb R)$. In the following, we identify $N(\mathbb R)$ with $F_{\infty}^{n-1}$ via the isomorphism 
\begin{align}\label{isom N}
\small{
F_{\infty}^{n-1} \overset{\sim}{\rightarrow} N(\mathbb R);~ (\xi_1, \dots , \xi_{n-1}) \mapsto  
\left(  \begin{matrix} 
     ~&~& ~ & \xi_1 \\
     ~&I&~ & \vdots \\
     ~&~& ~ & \xi_{n-1} \\
     0&\cdots& 0 & 1 \\      
   \end{matrix}\right)},
\end{align}
which clearly induces an isomorphism $\mathfrak{a}_1^{-1}\mathfrak{a}_n \oplus \cdots \oplus \mathfrak{a}_{n-1}^{-1}\mathfrak{a}_n \overset{\sim}{\rightarrow} N$. 
We also identify $N$ with the $\mathcal O_F$-lattice $\mathfrak{a}_1^{-1}\mathfrak{a}_n \oplus \cdots \oplus \mathfrak{a}_{n-1}^{-1}\mathfrak{a}_n$ in $F_{\infty}^{n-1}$ via this isomorphism. We denote by $d\xi$ the Haar measure on $N(\mathbb R)\simeq F_{\infty}^{n-1}$ which is the product measure of $dx_F$ on $F_{\infty}$ defined in Section \ref{haar meas}. 
We denote by 
\begin{equation}
N^{\vee} :=\{y \in F_{\infty}^{n-1} \mid \braket{x,y} \in \mathbb Z, ~ \forall x \in N \} = \mathfrak a_1 \mathfrak a_n^{-1} \mathfrak d_F^{-1} \oplus \cdots \oplus \mathfrak a_{n-1} \mathfrak a_n^{-1} \mathfrak d_F^{-1}
\end{equation} 
the dual lattice of $N$ with respect to the pairing $\braket{~,~}: F_{\infty}^{n-1} \times F_{\infty}^{n-1} \rightarrow \mathbb R$.

\begin{dfn}[Fourier coefficients]
For  $* \in \{\data{\mathfrak a}{L}, (L,[\mathfrak a])\}$ and $\nu \in N^{\vee}$, we define the $\nu$-th Fourier coefficient of $E_{*}(z,s)$ as
\begin{equation}
\mathscr I_{*,\nu}(z,s) := \frac{1}{vol(N\bs N(\mathbb R))} \int_{N \bs N(\mathbb R)} E_{*}(\xi z,s) e^{-2\pi i \braket{\nu,\xi}} d\xi.
\end{equation}
\end{dfn}
Then, by the general theory of the Fourier expansion, we have 
\begin{equation}\label{fourier exp eqn}
E_{*}(z,s) = \sum_{\nu \in N^{\vee}} \mathscr I_{*,\nu}(z,s).
\end{equation}

\paragraph{Iwasawa normal form}
In order to state the Fourier expansion formula explicitly, for $z  \in \mathfrak h_F^n$, we always take a representative $g  \in G(\mathbb R)=GL_n(F_{\infty})$ of $z$ (i.e.,  $z=[g]$) of the following form called the Iwasawa normal form: $g=XY$ with 
\begin{align}
\small{X = \left(  \begin{matrix} 
     1& & x_{ij}  \\
     ~&\ddots&    \\
     \hsymb{0}&~& 1   \\
   \end{matrix}\right),  \quad~
Y =  \left(  \begin{matrix} 
     y'_{1}&~& ~ & \hsymb{0} \\
     ~&\ddots& ~ & ~ \\
     ~&~& y'_{n-1} & ~ \\
     \hsymb{0}&~& ~ & 1 \\      
   \end{matrix}\right),}
\end{align}
where $x_{ij} \in F_{\infty}$, and $y'_{i}=y_{i}\cdots y_{n-1}$ for $y_{i} \in T_F=\prod_{\sigma}\mathbb R_{>0}$. 
Note that the existence and the uniqueness of the Iwasawa normal form is guaranteed by the Iwasawa decomposition of $GL_n(\mathbb R)$ and $GL_n(\mathbb C)$.
With the above notation, for $2\leq j \leq n$, we set $\mathbf x_{j}:={}^t\!(x_{1j}, \dots, x_{j-1, j}) \in F_{\infty}^{j-1}$, which is essentially the $j$-th column of $X$. 
Furthermore, for $1 \leq j \leq n-1$, we put 
\begin{align}
\small{
X^{(j)} :=  \left(  \begin{matrix} 
     1&x_{12}& \cdots &x_{1,n-j}  \\
     ~&\ddots& \ddots & \vdots \\
     ~&~& \ddots & x_{n-j-1,n-j} \\
     \hsymb{0}&~& ~ & 1 \\      
   \end{matrix}\right),  \quad~
Y^{(j)} :=  \left(  \begin{matrix} 
     y_{1}\cdots y_{n-j}&~&~&  \hsymb{0} \\
     ~&y_{2}\cdots y_{n-j}&~&~ \\
     ~&~&\ddots& ~  \\
      \hsymb{0}&~&~& y_{n-j}  \\
   \end{matrix}\right),}
\end{align}
and set $g^{(j)}:= X^{(j)}Y^{(j)} \in GL_{n-j}(F_{\infty})$.
We fix these notations throughout the paper.

\paragraph{Some arithmetic functions and Bessel function}
We prepare some arithmetic functions and Bessel function. 

Let $\mathfrak m \subset \mathcal O_F$ be a non-zero integral ideal. We define 
\begin{enumerate}[--]
\item (Euler's totient function) $\varphi(\mathfrak m) := \#(\mathcal O_F /\mathfrak m)^{\times}$,
\item (Divisor sum) 
$\sigma_{s}(\mathfrak m,\chi) := \sum_{\mathfrak n |\mathfrak m}\chi(\mathfrak n) N \mathfrak n^s$, \quad ($\chi \in \Hom(Cl_F, \mathbb C^{\times})$, $s \in \mathbb C$). 
\end{enumerate}
Furthermore, let $\mathfrak b \subset F$ be a fractional $\mathcal O_F$-ideal such that $\mathfrak b \subset \mathfrak d_F^{-1}$. Then, since $\mathcal O_F$ is a Dedekind domain, $\mathcal O_F/\mathfrak m$ is a principal ideal ring, and thus $\mathfrak b \mathfrak m^{-1}/ \mathfrak b$ is isomorphic to $\mathcal O_F/\mathfrak m$ as an $\mathcal O_F/\mathfrak m$-module. We denote by $(\mathfrak b \mathfrak m^{-1}/ \mathfrak b)^{\times}$ the set of generators of $\mathfrak b \mathfrak m^{-1}/ \mathfrak b$ as an $\mathcal O_F/\mathfrak m$-module. We define the special case of the Gauss sum as
\begin{enumerate}[--]
\item (Ramanujan sum) $\displaystyle \tau(\mathfrak m, \mathfrak b) :=\sum_{x \in (\mathfrak b \mathfrak m^{-1}/ \mathfrak b)^{\times}} e^{2\pi i Tr_{F/ \mathbb Q}(x)}$,
\end{enumerate}
where the sum is taken over a system of representatives of $(\mathfrak b \mathfrak m^{-1}/ \mathfrak b)^{\times}$. This is well-defined since $\mathfrak b \subset \mathfrak d_F^{-1}$.
\begin{lem}\label{arith func lem}
Let $\mathfrak m \subset \mathcal O_F$ be a non-zero integral ideal, and let $\mathfrak b \subset F$ be a fractional $\mathcal O_F$-ideal such that $\mathfrak b \subset \mathfrak d_F^{-1}$, then we have
\begin{align}
\sum_{\mathfrak a | \mathfrak m} \varphi(\mathfrak a) =N\mathfrak m,  \quad
\sum_{\mathfrak a | \mathfrak m} \tau(\mathfrak a, \mathfrak b) = 
\begin{cases}
N\mathfrak m &(\mathfrak m | \mathfrak b \mathfrak d_F)\\
0 &(otherwise)
\end{cases}
\end{align}
\end{lem}
\begin{proof}
This follows directly from the identity $\mathfrak b \mathfrak m^{-1} /\mathfrak b = \coprod_{\mathfrak a |\mathfrak m} (\mathfrak b \mathfrak a^{-1} /\mathfrak b)^{\times}$.
\end{proof}

\begin{dfn}
Let $\mathbf s=(s_1, \dots, s_k) \in \mathbb C^k$, $\mathbf t=(t_1, \dots, t_l) \in \mathbb C^l$ be tuples of complex variables such that $\re{s_i}, \re{t_j} >2$ for all $i,j$, and let $\mathfrak b_j \subset \mathfrak d_F^{-1}$ $1\leq j\leq l$ be fractional $\mathcal O_F$-ideals. For $\mathscr A\in Cl_F$, we define
\begin{align}
Z_{\mathscr A}(\mathbf s; \mathbf t; (\mathfrak b_j)_j):=
\sum_{\substack{\mathfrak m_1,\dots, \mathfrak m_k \subset \mathcal O_F\\ \mathfrak n_1,\dots, \mathfrak n_l \subset \mathcal O_F\\ \mathfrak m_1\cdots \mathfrak m_k \mathfrak n_1\cdots \mathfrak n_l \in \mathscr A}} 
\prod_{i=1}^k \frac{\varphi(\mathfrak m_i)}{N\mathfrak m_i^{s_i}} 
\prod_{j=1}^l\frac{\tau(\mathfrak n_j, \mathfrak b_j)}{N\mathfrak n_j^{t_j}},
\end{align}
if $k,l\geq 1$. In the case where $k=0$ or $l=0$, we define
\begin{align}
&Z_{\mathscr A}(\emptyset, \emptyset):=
\begin{cases}
1& (\mathscr A = [\mathcal O_F]) \nonumber \\
0& (\mathscr A \neq [\mathcal O_F])
\end{cases}, \quad
Z_{\mathscr A}(\mathbf s; \emptyset):=
\sum_{\substack{\mathfrak m_1,\dots, \mathfrak m_k \subset \mathcal O_F\\ \mathfrak m_1\cdots \mathfrak m_k \in \mathscr A}} 
\prod_{i=1}^k \frac{\varphi(\mathfrak m_i)}{N\mathfrak m_i^{s_i}}, \\
&Z_{\mathscr A}(\emptyset; \mathbf t; (\mathfrak b_j)_j):=
\sum_{\substack{\mathfrak n_1,\dots, \mathfrak n_l \subset \mathcal O_F\\ \mathfrak n_1\cdots \mathfrak n_l \in \mathscr A}} 
\prod_{j=1}^l\frac{\tau(\mathfrak n_j, \mathfrak b_j)}{N\mathfrak n_j^{t_j}}.
\end{align}
Furthermore, for a character $\chi \in \Hom(Cl_F, \mathbb C^{\times})$, we define
\begin{align}
Z(\mathbf s; \mathbf t; (\mathfrak b_j)_j; \chi):= \sum_{\mathscr A \in Cl_F}\chi(\mathscr A)Z_{\mathscr A}(\mathbf s; \mathbf t; (\mathfrak b_j)_j).
\end{align}
\end{dfn}

\begin{lem}\label{z_chi lem}
Let the notations be as above. Then $Z_{\mathscr A}(\mathbf s; \mathbf t; (\mathfrak b_j)_j)$ converges absolutely and compactly for $\re{s_i}, \re{t_j}>2$. Moreover, for $\chi \in \Hom(Cl_F,\mathbb C^{\times})$, we have
\begin{align}
Z(\mathbf s; \mathbf t; (\mathfrak b_j)_j; \chi) = 
\prod_{i=1}^k \frac{L(s_i-1,\chi)}{L(s_i,\chi)} \prod_{j=1}^l \frac{\sigma_{1-t_j}(\mathfrak b_j \mathfrak d_F, \chi)}{L(t_j,\chi)},
\end{align}
where $L(s,\chi)$ is the Hecke $L$-function associated to the character $\chi$.
\end{lem}
\begin{proof}
This follows from Lemma \ref{arith func lem}.
\end{proof}
On the other hand, we have
\begin{align}
Z_{\mathscr A}(\mathbf s; \mathbf t; (\mathfrak b_j)_j)=\frac{1}{h_F}\sum_{\chi \in \Hom(Cl_F, \mathbb C^{\times})} \chi(\mathscr A^{-1})Z(\mathbf s; \mathbf t; (\mathfrak b_j)_j; \chi).
\end{align}
Therefore we get the following:

\begin{cor}\label{z_a cor} 
For $\mathscr A\in Cl_F$, we have
\begin{align}
Z_{\mathscr A}(\mathbf s; \mathbf t; (\mathfrak b_j)_j)=\frac{1}{h_F}\sum_{\chi \in \Hom(Cl_F, \mathbb C^{\times})} \chi(\mathscr A^{-1})\prod_{i=1}^k \frac{L(s_i-1,\chi)}{L(s_i,\chi)} \prod_{j=1}^l \frac{\sigma_{1-t_j}(\mathfrak b_j \mathfrak d_F, \chi)}{L(t_j,\chi)}.
\end{align}
In particular $Z_{\mathscr A}(\mathbf s; \mathbf t; (\mathfrak b_j)_j)$ can be continued meromorphically to whole $s_i, t_j \in \mathbb C$. \qed
\end{cor}

In this section, we use a slightly different gamma factor from those in Section \ref{integral formula}. 
For a number field $k/\mathbb Q$ of degree $d$, define
\begin{equation}\label{abs gamma}
\Gamma_k(s) :=\prod_{\sigma \in S_k} \Gamma (\frac{n_{\sigma}s}{2}).
\end{equation}
Let $K_s(x)$ ($s \in \mathbb C$, $x \in \mathbb R_{>0}$) be the $K$-Bessel function, that is, 
\begin{equation}\label{K-bessel}
K_s(x)= \frac{1}{2}\int_{0}^{\infty}e^{-\frac{1}{2}x(u+\frac{1}{u})}u^s\frac{du}{u}.
\end{equation}
Then, for $s \in \mathbb C$ and $x=(x_{\sigma})_{\sigma} \in T_k$, we define the $K$-Bessel function $K_k(s,x)$ over $k$ as
\begin{equation}\label{bessel over k}
K_k(s,x):=\prod_{\sigma \in S_k} K_{\frac{n_{\sigma}s}{2}}(x_{\sigma}).
\end{equation}

\subsection{Statement of the Fourier expansion formula}
We have the following formulas for the Fourier coefficients $\mathscr I_{L,\mathfrak [a],\nu}(z,s)$. 
For $1 \leq j \leq n-1$, we set $L^{(j)} :=\mathfrak a_1\oplus \cdots \oplus \mathfrak a_{n-j} \subset F^{n-j}$. 
\begin{thm}\label{fourier coeff thm}
\begin{enumerate}[{\rm (1)}] 
\item (Constant term, $\nu=0$) For $\re{s}>1$, we have
\begin{align}
&\frac{\mathscr I_{L,[\mathfrak a],0}(z,s)}{\zeta_F(\mathfrak a^{-1},ns)}
= \delta_{[\mathfrak a], [\mathfrak a_n]} (N \mathfrak a_n)^{-ns} \prod_{\sigma} |\det g^{(1)}_{\sigma}|^{n_{\sigma}s}  \\
&+ \frac{2^{r_2(F)} \pi^{\frac{d}{2}}}{\sqrt{|d_F|} N \mathfrak a_n} \frac{\Gamma_F(ns-1)}{\Gamma_F(ns)} 
\prod_{\sigma} |\det g^{(1)}_{\sigma}|^{n_{\sigma}\frac{1-s}{n-1}}
\sum_{\substack{\mathfrak m \subset \mathcal O_F\\ \mathfrak m \neq 0}} \frac{\varphi (\mathfrak m)}{N \mathfrak m^{ns}} 
\frac{E_{L^{(1)},[\mathfrak a \mathfrak m^{-1}]}([g^{(1)}],\frac{ns-1}{n-1})}{\zeta_F(\mathfrak a^{-1}\mathfrak m,ns-1)},\nonumber
\end{align}
where $\delta$ is the Kronecker delta, $g^{(1)}_{\sigma}$ is the $\sigma$-component of $g^{(1)}$ and we identify $[g^{(1)}] \in \mathfrak h_F^{n-1}$, and $E_{L^{(1)},[\mathfrak a \mathfrak m^{-1}]}([g^{(1)}],s)$ is the Eisenstein series on $\mathfrak h_F^{n-1}$.
\item (Non-constant terms, $\nu \neq 0$) For $\re{s}>1$, we have
\begin{align}
\sum_{\nu \in N^{\vee} - \{0\}} \frac{\mathscr I_{L,\mathfrak [a],\nu}(z,s)}{\zeta_F(\mathfrak a^{-1},ns)} = \frac{2^{d} \pi^{\frac{dns}{2}}}{\sqrt{|d_F|}(N \mathfrak a_n)^{ns}} 
\sum_{\substack{\mathfrak n \subset \mathcal O_F\\ \mathfrak n \neq 0}} N(\mathfrak n \mathfrak d_F^{-1})^{ns-1} \sum_{\mathfrak m \in [\mathfrak a \mathfrak a_n^{-1} \mathfrak n \mathfrak d_F^{-1}]} \frac{\tau(\mathfrak m, \mathfrak n \mathfrak d_F^{-1})}{N \mathfrak m ^{ns}}& \nonumber \\
\times \sum_{\nu \in (N^{\vee})_{\mathfrak n^{-1}}} e^{2 \pi i \braket{\nu, \mathbf x_{n}}}
\frac{K_F(ns-1,(2 \pi ||\nu g^{(1)}_{\sigma}||)_{\sigma})}{\Gamma_F(ns)} 
\prod_{\sigma}\frac{|\det g^{(1)}_{\sigma}|^{n_{\sigma}s}}{||\nu g^{(1)}_{\sigma}||^{\frac{n_{\sigma}}{2}(ns-1)}}.&
\end{align}
\end{enumerate}
\end{thm}
We give a proof of this theorem in Section \ref{proof sect}. 

\begin{rmk}
We can deduce the corresponding results of Epstein~\cite{epstein03}, Liu-Masri~\cite{liumasri15} and Yamamoto~\cite{yamamoto08} from this theorem. We go further to get rid of the Eisenstein series in the constant term. Note that in order to deduce the result in \cite{epstein03} and \cite{liumasri15}, we use the functional equation of the Eisenstein series, and specialize this formula to $s=0$ (see also Theorem \ref{phi psi conti} below).
\end{rmk}

\paragraph{Fourier expansion formula}
By using Theorem \ref{fourier coeff thm} recursively, we obtain a kind of Fourier expansion formula for $E_{L, [\mathfrak a]}(z,s)$.
In order to simplify the presentation, we first define some additional notation. 
For $0\leq j\leq n-1$, we set
\begin{align}
c_j(z,s):=& 
\left(\frac{2^{r_2(F)}\pi^{\frac{d}{2}}}{\sqrt{|d_F|}}\right)^j \frac{\Gamma_F(ns-j)}{\Gamma_F(ns)} 
\frac{(N\mathfrak a_i)^{j-ns}}{N\mathfrak a_n\cdots N\mathfrak a_{n-j+1}} \nonumber \\
&\quad \quad \quad \quad \times \prod_{k=1}^{n-j-1}|N_{F/\mathbb Q}(y_k)|^{ks}\prod_{k=n-j}^{n-1}|N_{F/\mathbb Q}(y_k)|^{(n-k)(1-s)}, \\
d_j(z, s):=& \frac{2^d \pi^{\frac{ds}{2}(n-j)}}{\sqrt{|d_F|} \Gamma_F(ns-j)}c_j(z, s), 
\end{align}
where we define $c_0(z,s):= \prod_{k=1}^{n-1}|N_{F/\mathbb Q}(y_k)|^{ks}$ for $j=0$. 
For $1 \leq j \leq n-1$, set $\Lambda^{(j)}:= \mathfrak a_1 \mathfrak a_{n-j+1}^{-1}\mathfrak d_{F}^{-1}\oplus \cdots \oplus \mathfrak a_{n-j}\mathfrak a_{n-j+1}^{-1}\mathfrak d_{F}^{-1}$. Note that $\Lambda^{(1)}=N^{\vee}$.

\begin{dfn}
\begin{enumerate}[{\rm (1)}]
\item For $0\leq j \leq n-1$, $\re{s}>1$, we define
\begin{align}
\varPhi_j(z,s) := c_j(z,s) Z_{[\mathfrak a \mathfrak a_{n-j}^{-1}]} (\,\overbrace{ns, ns-1, \dots , ns-j+1}^{j-tuple}; \emptyset),
\end{align}
where we assume $(ns, ns-1 \dots, ns-j+1)$ is the empty tuple $\emptyset$ if $j=0$.
\item For $0\leq j \leq n-2$, $\re{s}>1$, we define
\begin{multline}
\varPsi_j (z,s) := 
d_j(z,s) \sum_{\substack{\mathfrak n \subset \mathcal O_F\\ \mathfrak n \neq 0}} 
Z_{[\mathfrak a \mathfrak a_{n-j}^{-1}\mathfrak n \mathfrak d_F^{-1}]} (\,\overbrace{ns, ns-1, \dots , ns-j+1}^{j-tuple}; ns-j; \mathfrak n \mathfrak d_F^{-1}) \\
\times N(\mathfrak n \mathfrak d_F^{-1})^{ns-j-1} 
\sum_{\nu \in (\Lambda^{(j+1)})_{\mathfrak n^{-1}}} e^{2 \pi i \braket{\nu, \mathbf x_{n-j}}}
\frac{K_F(ns-j-1,(2 \pi ||\nu g^{(j+1)}_{\sigma}||)_{\sigma})}{\prod_{\sigma}||\nu g^{(j+1)}_{\sigma}||^{\frac{n_{\sigma}}{2}(ns-j-1)}},
\end{multline}
where we assume $(ns, ns-1 \dots, ns-j+1) = \emptyset$ if $j=0$, as above. 
\end{enumerate}
\end{dfn}

\begin{thm}\label{phi psi conti}
The functions $\varPhi_j(z,s)$ and $\varPsi_j(z,s)$ can be continued meromorphically to whole $s \in \mathbb C$. 
Furthermore, $\varPhi_j(z,s)$ (resp. $\varPsi_j(z,s)$) is holomorphic outside the poles of $\frac{L(ns-j,\chi)}{L(ns,\chi)}$ (resp. $L(ns, \chi)^{-1}$) for $\chi \in \Hom(Cl_F,\mathbb C^{\times})$ and $0 \leq j \leq n-1$. 
\end{thm}
The proof of this Theorem is given in Section \ref{phi psi proof}

\begin{thm}[Fourier expansion formula]\label{fourier exp thm}
For $\re{s}>1$, we have
\begin{align}
\frac{E_{L,[\mathfrak a]}(z,s)}{\zeta_F(\mathfrak a^{-1},ns)}= \sum_{j=0}^{n-2}(\varPhi_j(z,s)+\varPsi_j(z,s))+\varPhi_{n-1}(z,s).
\end{align}
In particular, $E_{L,[\mathfrak a]}(z,s)$ can be continued meromorphically to whole $s \in \mathbb C$
\end{thm}
\begin{proof}
This follows from (\ref{fourier exp eqn}), Theorem \ref{fourier coeff thm}, and Theorem \ref{phi psi conti}.
\end{proof}


\begin{ex}\label{phi psi ex}
In the case where $F=\mathbb Q$, $n=2$, $L=\mathbb Z \oplus \mathbb Z$, and $\mathfrak a=\mathbb Z$, we have
\begin{align}
&\varPhi_0(z,s)=y_1^s, \quad
\varPhi_1(z,s)=\sqrt{\pi}y_1^{1-s} \frac{\Gamma(s-\frac{1}{2})}{\Gamma(s)}\frac{\zeta_{\mathbb Q}(2s-1)}{\zeta_{\mathbb Q}(2s)},\\
&\varPsi_0(z,s)=\frac{2\pi \sqrt{y_1}}{\Gamma(s)\zeta_{\mathbb Q}(2s)}
\sum_{n\in \mathbb Z, \neq 0} |n|^{s-\frac{1}{2}}\sigma_{1-2s}(n)K_{s-\frac{1}{2}}(2\pi|n|y_1)e^{2\pi i n x_{12}},
\end{align}
where $\sigma_s(n)$ is the usual divisor sum. In this case, Theorem \ref{fourier exp thm} is nothing but the classical Fourier expansion of the real analytic Eisenstein series \cite[Theorem 3.1.8]{goldfeld06}.
\end{ex}


\subsection{Proof of Theorem \ref{fourier coeff thm}}\label{proof sect}
By Proposition \ref{prop eis ser} (3), it suffices to compute the Fourier coefficients $ \mathscr I_{\data{\mathfrak a}{L},\nu}(z,s)$. 
Put $\overline L := L^{(1)}$, $\overline g =(\overline g_{\sigma})_{\sigma}:= g^{(1)}$ for simplicity.

We fix a fundamental domain of $N \bs N(\mathbb R)$. Set
\begin{align}
L_{\mathfrak a, i}&:=\{x =(x_j)_j \in L_{\mathfrak a}\subset F^n \mid x_j=0 ~(\forall j\leq i-1), x_i \neq 0\}, &(i \leq n),\\
\overline L_{i}&:=\{x =(x_j)_j \in \overline L \subset F^{n-1} \mid x_j=0 ~(\forall j\leq i-1), x_i \neq 0\}, &(i \leq n-1), \\
N_i(\mathbb R) &:=\left\{\xi=(\xi_1, \dots, \xi_{n-1}) \in N(\mathbb R) ~\middle| ~ \xi_j =0 ~(j\neq i) \right\}, &(i \leq n-1), \\
N^i(\mathbb R) &:=\left\{\xi=(\xi_1, \dots, \xi_{n-1}) \in N(\mathbb R) ~\middle| ~ \xi_i =0  \right\}, &(i \leq n-1),\\
N_i &:= N \cap N_i(\mathbb R), \quad N^i :=N \cap N^i(\mathbb R),  &(i \leq n-1).
\end{align}
Here we identify matrices $\xi \in N(\mathbb R)$ with vectors $(\xi_1, \dots, \xi_{n-1}) \in F_{\infty}^{n-1}$ via the identification (\ref{isom N}).
We have $L_{\mathfrak a} = \coprod_{i=1}^nL_{\mathfrak a,i}$, $N=N_iN^i$, and if $i \leq n-1$, $N_i$ acts freely on $L_{\mathfrak a,i}$ by the matrix action from the right. 

\vspace{2mm}
\noindent \underline{\bf Step 1.} (Decomposition of the integral)
For $x \in F_{\infty}^n$, put $f(x):= \prod_{\sigma} \frac{|\det g_{\sigma}|^{n_{\sigma}s}}{||x g_{\sigma}||^{nn_{\sigma}s}}$. Then we can decompose the integral $\mathscr I_{\data{\mathfrak a}{L},\nu}(z,s)$ as 
\begin{align}
\mathscr I_{\data{\mathfrak a}{L},\nu}(z,s)=\frac{1}{vol(N \bs N(\mathbb R))} \sum_{i_0=1}^n \sum_{x \in \mathcal O_F^{\times} \bs L_{\mathfrak a,i_0}} \int_{N \bs N(\mathbb R)}f(x\xi)e^{-2\pi i \braket{\nu,\xi}}d\xi.
\end{align}
Set $\mathscr I_{\nu,i_0}:= \sum_{x \in \mathcal O_F^{\times} \bs L_{\mathfrak a,i_0}} \int_{N \bs N(\mathbb R)}f(x\xi)e^{-2\pi i \braket{\nu,\xi}}d\xi$. 
In the case where $i_0=n$, we easily see that $\mathscr I_{\nu,n}=0$ unless $\nu=0$ and $[\mathfrak a]=[\mathfrak a_n]$, in which case we have
\begin{align}
\mathscr I_{\nu,n}=
vol(N \bs N(\mathbb R)) \left(\frac{N \mathfrak a}{N \mathfrak a_n} \right)^{ns} \prod_{\sigma} |\det \overline g_{\sigma}|^{n_{\sigma}s}.
\end{align}
In the following, we assume $i_0 \leq n-1$. Then we calculate as
\begin{align}
\mathscr I_{\nu,i_0}
&= \sum_{x \in \mathcal O_F^{\times} \bs L_{\mathfrak a, i_0}/N_{i_0}} \sum_{\xi' \in N_{i_0}} \int_{N\bs N(\mathbb R)} f(x\xi' \xi) e^{-2\pi i \braket{\nu,\xi' \xi}} d\xi. \label{eq420} \\
&= \sum_{x \in \mathcal O_F^{\times} \bs L_{\mathfrak a, i_0}/N_{i_0}} \int_{N^{i_0}\bs N^{i_0}(\mathbb R)}\left( \int_{N_{i_0}(\mathbb R)} f(x\xi)e^{-2\pi i \braket{\nu,\xi}} d\xi_{i_0}\right) d\xi^{i_0},
\end{align}
where $\xi' \xi$ is the multiplication as matrices, and we use $e^{-2\pi i \braket{\nu, \xi' \xi}}=e^{-2\pi i \braket{\nu,\xi}}$ since $\braket{\nu,\xi'} \in \mathbb Z$. Furthermore, $d\xi_{i_0}$ and $d\xi^{i_0}$ are the Haar measures on $N_{i_0}(\mathbb R)$ and $N^{i_0}(\mathbb R)$ normalized in the same way as $d\xi$.

\vspace{2mm}
\noindent \underline{\bf Step 2.} (Calculation of the integral)

Now, $x\xi = (x_1, \dots, x_{n-1}, x_1\xi_1+\cdots + x_{n-1}\xi_{n-1}+x_n)$ for $\xi = (\xi_1 ,\dots, \xi_{n-1})$.
Since $x_{i_0} \neq 0$, we can replace $x_1\xi_1+\cdots + x_{n-1}\xi_{n-1}+x_n$ with $x_{i_0}\xi_{i_0}$ by the change of variables. Then we obtain
\begin{align}
 \int_{N^{i_0}\bs N^{i_0}(\mathbb R)}&\left( \int_{N_{i_0}(\mathbb R)} f(x\xi)e^{-2\pi i \braket{\nu,\xi}} d\xi_{i_0}\right) d\xi^{i_0} \\
=e^{2\pi i Tr_{F/\mathbb Q}(\frac{\nu_{i_0}}{x_{i_0}}x_n)} &\int_{N^{i_0}\bs N^{i_0}(\mathbb R)}e^{-2 \pi i Tr_{F/\mathbb Q}(\xi_1(\nu_1-\frac{\nu_{i_0}}{x_{i_0}}x_1) +\cdots + \xi_{n-1}(\nu_{n-1}-\frac{\nu_{i_0}}{x_{i_0}}x_{n-1}))}d\xi^{i_0} \nonumber \\
\times &\int_{N_{i_0}(\mathbb R)} f(x_1, \dots , x_{n-1}, x_{i_0}\xi_{i_0})e^{-2 \pi i Tr_{F/\mathbb Q}(\nu_{i_0}\xi_{i_0})}d\xi_{i_0}. \label{eq423}
\end{align}
Here the first integral in (\ref{eq423}) is $0$ unless $\nu = \frac{\nu_{i_0}}{x_{i_0}}(x_1, \dots, x_{n-1})$, in which case equal to $vol(N^{i_0}\bs N^{i_0}(\mathbb R))$. Therefore we assume $\nu = \frac{\nu_{i_0}}{x_{i_0}}(x_1, \dots, x_{n-1})$.
On the other hand, the second integral in (\ref{eq423}) can be written as
\begin{align}
\prod_{\sigma} |\det \overline g_{\sigma}|^{n_{\sigma}s} \int_{F_{\sigma}} \frac{e^{-2 \pi i Tr_{F_{\sigma}/\mathbb R}(\nu_{i_0}t_{\sigma})}}{\left(||\bar x \overline g_{\sigma}||^2 + |\bar x \mathbf x_{n,\sigma} + x_{i_0}t_{\sigma}|^2\right)^{nn_{\sigma}s/2}} dt_{\sigma},
\end{align}
where $\bar x=(x_1, \dots, x_{n-1})$ is the first $n-1$ components of $x$, and $\mathbf x_{n,\sigma} \in F_{\sigma}^{n-1}$ is the $\sigma$-component of $\mathbf x_n$, and $\bar x \mathbf x_{n,\sigma}$ is the scalar product. 
By an elementary computation, we have
\begin{align}
& \int_{F_{\sigma}} \frac{e^{-2 \pi i Tr_{F_{\sigma}/\mathbb R}(\nu_{i_0}t_{\sigma})}}{\left(||\bar x \overline g_{\sigma}||^2 + |\bar x \mathbf x_{n,\sigma} + x_{i_0}t_{\sigma}|^2\right)^{nn_{\sigma}s/2}} dt_{\sigma} \nonumber \\
&=
\begin{cases} \displaystyle
    n_{\sigma} \pi^{\frac{n_{\sigma}}{2}} 
    \frac{\Gamma(\frac{nn_{\sigma}s-n_{\sigma}}{2})}{\Gamma(\frac{nn_{\sigma}s}{2})}
    |\sigma (x_{i_0})|^{-n_{\sigma}} ||\bar x \overline g_{\sigma}||^{n_{\sigma}(1-ns)} & (\nu_{i_0}=0) \\
    2^{n_{\sigma}} \pi^{\frac{n_{\sigma}ns}{2}} 
    e^{2 \pi i Tr_{F_{\sigma}/\mathbb R}(\nu \mathbf x_{n,\sigma})}
    \frac{K_{\frac{nn_{\sigma}s-n_{\sigma}}{2}}(2\pi ||\nu  \overline g_{\sigma}||)}{\Gamma(\frac{nn_{\sigma}s}{2})} 
    \frac{|\sigma(\frac{\nu_{i_0}}{x_{i_0}})|^{n_{\sigma}(ns-1)}}{|\sigma (x_{i_0})|^{n_{\sigma}}} 
    ||\nu \overline g_{\sigma}||^{\frac{n_{\sigma}}{2}(1-ns)} & (\nu_{i_0} \neq 0)
  \end{cases}
\end{align}
\noindent \underline{\bf Step 3.} (Calculation of the summation)

We take a closer look at $L_{\mathfrak a, i_0}/ N_{i_0}$. 
For $x'=(x_1,\dots, x_{n-1}) \in F^{n-1}-\{0\}$, define
\begin{align}
\mathfrak b_{x'}&:=\{ \alpha \in F \mid \alpha x_i \in \mathfrak a_i, 1\leq \forall i \leq n-1\}, \\
S_{x'}&:=\{\alpha \in F \mid (x_1, \dots, x_{n-1}, \alpha) \in L_{\mathfrak a}\},
\end{align}
so that we have $Fx'\cap \overline L =\mathfrak b_{x'}x'$ and $L_{\mathfrak a, i_0} = \coprod_{x' \in \overline L_{i_0}} \{x'\} \times S_{x'}$.
Note that $\mathfrak b_{x'}$ is a fractional $\mathcal O_F$-ideal since $x' \neq 0$, and we have, for any integral ideal $\mathfrak m \subset \mathcal O_F$, 
\begin{align}\label{eq429}
\{x' \in \overline L -\{0\} \mid \mathfrak a \mathfrak b_{x'}^{-1} = \mathfrak m\} = \overline L_{\mathfrak a \mathfrak m^{-1}},
\end{align}
\begin{lem}\label{lem1 in pf}
We have $S_{x'}=\emptyset$ unless $\mathfrak a \subset  \mathfrak b_{x'}$, in which case we have
\begin{enumerate}[{\rm (i)}]
\item  $S_{x'}=\{\alpha \in \mathfrak a_n \mathfrak a^{-1} \mid  v_{\mathfrak p}(\alpha) = v_{\mathfrak p}(\mathfrak a_n \mathfrak a^{-1}) ~\forall \mathfrak p \text{\rm : prime ideal } \mathfrak p | \mathfrak a \mathfrak b_{x'}^{-1}\}$, where $v_{\mathfrak p}$ is the usual additive $\mathfrak p$-adic valuation. In particular, $\mathfrak a_n \mathfrak b_{x'}^{-1}$ acts on $S_{x'}$ by addition.
\item $S_{x'}/\mathfrak a_n \mathfrak b_{x'}^{-1} = (\mathfrak a_n \mathfrak a^{-1}/\mathfrak a_n \mathfrak b_{x'}^{-1})^{\times}$: the set of generators of $\mathfrak a_n \mathfrak a^{-1}/\mathfrak a_n \mathfrak b_{x'}^{-1}$ as an $\mathcal O_F/\mathfrak a \mathfrak b_{x'}^{-1}$-module. 
\item $L_{\mathfrak a, i_0}/N_{i_0} = \coprod_{x' \in \overline L_{i_0}} \{x'\} \times (S_{x'}/x_{i_0} \mathfrak a_n\mathfrak a_{i_0}^{-1})$, where $x'=(x_1,\dots, x_{n-1})$.
\end{enumerate}
\end{lem}
\begin{proof}
The first assertion and (i) follows from the fact
\begin{align}
\alpha \in S_{x'} \Leftrightarrow  \mathfrak b_{x'} \cap \frac{1}{\alpha}\mathfrak a_n =\mathfrak a
\Leftrightarrow \mathcal O_F \cap \frac{1}{\alpha}\mathfrak a_n \mathfrak b_{x'}^{-1} = \mathfrak a \mathfrak b_{x'}^{-1}, \quad \forall \alpha \in F^{\times}.
\end{align}
(ii) follows directly from (i), and (iii) is obvious.
\end{proof}

Note that $\nu_{i_0}=0$ if and only if $\nu =0$ by the assumption $\nu = \frac{\nu_{i_0}}{x_{i_0}}(x_1, \dots, x_{n-1})$. Therefore, in the case $\nu =0$, we obtain
\begin{align}
&\mathscr I_{0,i_0} = \sum_{x' \in \mathcal O_F^{\times}\bs \overline L_{i_0}} \sum_{\alpha \in S_{x'}/x_{i_0} \mathfrak a_n\mathfrak a_{i_0}^{-1}} 
\frac{vol(N^{i_0}\bs N^{i_0}(\mathbb R))}{|N_{F/\mathbb Q}(x_{i_0})|} 
2^{r_2(F)}\pi^{\frac{d}{2}}  \frac{\Gamma_F(ns-1)}{\Gamma_F(ns)}
 \prod_{\sigma}  \frac{|\det \overline{g}_{\sigma}|^{n_{\sigma}s}}{||x' \overline g_{\sigma}||^{n_{\sigma}(ns-1)}} \nonumber \\ 
&=  \frac{vol(N\bs N(\mathbb R))}{\sqrt{|d_F|}} 2^{r_2(F)}\pi^{\frac{d}{2}}  \frac{\Gamma_F(ns-1)}{\Gamma_F(ns)} \left(\frac{N \mathfrak a}{N \mathfrak a_n}\right)
\sum_{\substack{x' \in \mathcal O_F^{\times}\bs \overline L_{i_0}\\ \mathfrak a \mathfrak b_{x'}^{-1} \subset \mathcal O_F}} \frac{\varphi(\mathfrak a \mathfrak b_{x'}^{-1})}{N(\mathfrak a \mathfrak b_{x'}^{-1})}  
\prod_{\sigma}  \frac{|\det \overline{g}_{\sigma}|^{n_{\sigma}s}}{||x' \overline g_{\sigma}||^{n_{\sigma}(ns-1)}},
\end{align}
where we use 
\begin{align}
\frac{\#(S_{x'}/x_{i_0} \mathfrak a_n\mathfrak a_{i_0}^{-1})}{|N_{F/\mathbb Q}(x_{i_0})|vol(N_{i_0}\bs N_{i_0}(\mathbb R))}
&= \frac{\varphi(\mathfrak a \mathfrak b_{x'}^{-1})}{N(\mathfrak a_n\mathfrak b_{x'}^{-1})\sqrt{|d_F|}}. 
\end{align}
By taking the summation over $1 \leq i_0 \leq n-1$, and using (\ref{eq429}) and Proposition \ref{prop eis ser} (3), we obtain Theorem \ref{fourier coeff thm} (1). 

Similarly, in the cases where $\nu \neq 0$, we obtain
\begin{align}\label{eq438}
\mathscr I_{\nu,i_0} =&vol(N\bs N(\mathbb R)) 2^d \pi^{\frac{dns}{2}} e^{2\pi i \braket{\nu,\mathbf x_n}} \frac{K_F(ns-1,(2\pi ||\nu \overline g_{\sigma}||)_{\sigma})}{\Gamma_F(ns)}
\prod_{\sigma}\frac{|\det \overline g_{\sigma}|^{n_{\sigma}s}}{||\nu \overline g_{\sigma}||^{\frac{n_{\sigma}}{2}(ns-1)}} \nonumber \\
&\times \sum_{\substack{x' \in \mathcal O_F^{\times} \bs \overline L_{i_0}\\ \nu=\frac{\nu_{i_0}}{x_{i_0}}x'}}
\frac{|N_{F/\mathbb Q}(\frac{\nu_{i_0}}{x_{i_0}})|^{ns-1}}{|N_{F/\mathbb Q}(x_{i_0})|vol(N_{i_0}\bs N_{i_0}(\mathbb R))}
\sum_{\alpha \in S_{x'}/x_{i_0} \mathfrak a_n\mathfrak a_{i_0}^{-1}} e^{2\pi i Tr_{F/\mathbb Q}(\frac{\nu_{i_0}}{x_{i_0}}\alpha)}.
\end{align}
Now, suppose $\nu \in (N^{\vee})_{\mathfrak n^{-1}}$ for an integral ideal $\mathfrak n \subset \mathcal O_F$ (and $\nu_{i_0}\neq 0$). Then we claim that the last summation in (\ref{eq438}) is $0$ unless $\mathfrak a \mathfrak b_{x'}^{-1} \subset \mathcal O_F$, in which case we have
\begin{align}
\sum_{\alpha \in S_{x'}/x_{i_0} \mathfrak a_n\mathfrak a_{i_0}^{-1}} e^{2\pi i Tr_{F/\mathbb Q}(\frac{\nu_{i_0}}{x_{i_0}}\alpha)}
=\#(\mathfrak a_n \mathfrak b_{x'}^{-1}/x_{i_0} \mathfrak a_n\mathfrak a_{i_0}^{-1}) \tau(\mathfrak a \mathfrak b_{x'}^{-1}, \mathfrak n \mathfrak d_F^{-1}). \label{eqn a}
\end{align} 
Indeed, by (\ref{eq429}), we have $(N^{\vee})_{\mathfrak n^{-1}}=\{\nu \in N^{\vee} \mid  \mathfrak a_n \mathfrak b_{\nu}^{-1} \mathfrak d_F=\mathfrak n\}$, and therefore, for $x' \in \overline L_{i_0}$ such that $\nu=\frac{\nu_{i_0}}{x_{i_0}}x'$, we have 
\begin{align}
\frac{\nu_{i_0}}{x_{i_0}} \mathfrak a_n \mathfrak b_{x'}^{-1}=\mathfrak a_n \mathfrak b_{\nu}^{-1} = \mathfrak n\mathfrak d_F^{-1} \subset \mathfrak d_F^{-1}.
\end{align}
Now, the claim follows easily from Lemma \ref{lem1 in pf}.
Thus, for $\nu \in (N^{\vee})_{\mathfrak n^{-1}}$, we obtain $\mathscr I_{\nu,i_0} =0$ unless $\nu_{i_0}\neq 0$ and $\nu_j=0$ for all $j < i_0$, in which case we have
\begin{multline}
\mathscr I_{\nu,i_0} 
=
\frac{vol(N\bs N(\mathbb R))}{\sqrt{|d_F|}} 2^d \pi^{\frac{dns}{2}} e^{2\pi i \braket{\nu,\mathbf x_n}} \frac{K_F(ns-1,(2\pi ||\nu \overline g_{\sigma}||)_{\sigma})}{\Gamma_F(ns)}
\prod_{\sigma}\frac{|\det \overline g_{\sigma}|^{n_{\sigma}s}}{||\nu \overline g_{\sigma}||^{\frac{n_{\sigma}}{2}(ns-1)}} \\ 
\times \left(\frac{N \mathfrak a}{N \mathfrak a_n}\right)^{ns} N(\mathfrak n \mathfrak d_F^{-1})^{ns-1}
\sum_{\substack{\mathfrak m \in [\mathfrak a \mathfrak a_n^{-1} \mathfrak n \mathfrak d_{F}^{-1}]\\ \mathfrak m \subset \mathcal O_F}}
\frac{\tau(\mathfrak m, \mathfrak n \mathfrak d_F^{-1})}{N\mathfrak m^{ns}}.
\end{multline}
Here we use the identity
\begin{align}
\left\{x' \in \mathcal O_F^{\times} \bs \overline L_{i_0} \middle| \nu=\frac{\nu_{i_0}}{x_{i_0}}x', \mathfrak a \mathfrak b_{x'}^{-1} = \mathfrak m \right\}
\overset{i_{0}\text{-th proj.}}{\overset{\sim}{\longrightarrow}} &
\left\{x_{i_0} \in \mathcal O_F^{\times} \bs (\mathfrak a_{i_0}-\{0\}) \middle| \frac{x_{i_0}}{\nu_{i_0}}\mathfrak a\mathfrak b_{\nu}^{-1}=\mathfrak m \right\} \nonumber \\
=\quad\quad& \left\{x_{i_0} \in \mathcal O_F^{\times} \bs F^{\times} \middle| \frac{x_{i_0}}{\nu_{i_0}}\mathfrak a\mathfrak b_{\nu}^{-1}=\mathfrak m \right\},
\end{align}
for any integral ideal $\mathfrak m \in [\mathfrak a \mathfrak a_n^{-1} \mathfrak n \mathfrak d_{F}^{-1}]$.
By taking the summation over $1\leq i_0 \leq n-1$ and $\nu \in N^{\vee}-\{0\}$, we obtain Theorem \ref{fourier coeff thm} (2).
This completes the proof. \qed 

\subsection{Proof of Theorem \ref{phi psi conti}}\label{phi psi proof}
The assertion for $\varPhi_j(z,s)$ follows directly from Corollary \ref{z_a cor}. To prove the assertion for $\varPsi_j(z,s)$, again by Corollary \ref{z_a cor}, we rewrite $\varPsi_j(z,s)$ as 
\begin{multline}\label{eqn459}
\varPsi_j(z,s)=d_j(z,s)(N\mathfrak d_F^{-1})^{ns-j-1} \frac{1}{h_F} \sum_{\chi \in \Hom(Cl_F,\mathbb C^{\times})} \frac{\chi(\mathfrak a^{-1}\mathfrak a_{n-j} \mathfrak d_F)}{L(ns, \chi)}\sum_{\mathscr A\in Cl_F}\chi(\mathscr A^{-1})\\
\times \sum_{\substack{\mathfrak n \in \mathscr A\\ \mathfrak n \subset \mathcal O_F}} \sum_{\nu \in (\Lambda^{(j+1)})_{\mathfrak n^{-1}}} (N\mathfrak n)^{ns-j-1} \sigma_{1+j-ns}(\mathfrak n, \chi)
\frac{K_F(ns-j-1, (2 \pi ||\nu g^{(j+1)}_{\sigma}||)_{\sigma})}{\prod_{\sigma} ||\nu g^{(j+1)}_{\sigma}||^{\frac{n_{\sigma}}{2}(ns-j-1)}} e^{2\pi i \braket{\nu, \mathbf x_{n-j}}}. 
\end{multline}
Therefore, it suffices to prove that the second row of (\ref{eqn459}) converges absolutely and compactly for $s \in \mathbb C$. 
We use the following asymptotic formulas.

\begin{lem}\label{asymp lem}
Let $R$ be any real number such that $R>1$. 
\begin{enumerate}[{\rm (1)}]
\item For any integral ideal $\mathfrak n$, any $\chi \in \Hom(Cl_F,\mathbb C^{\times})$, and any $s \in \mathbb C$ such that $|\re{s}|\leq R$, we have
\begin{align}
|\sigma_{s}(\mathfrak n, \chi)| \leq (N\mathfrak n)^{2R}.
\end{align}
\item There exists $C>0$ such that for any $x \in \mathbb R_{>0}$ and any $s\in \mathbb C$ such that $|\re{s}|\leq R$, we have
\begin{align}
\left|K_s(x) x^{-s}\right| \leq C  e^{-\frac{1}{2}x} x^{-10R}.
\end{align}
\item  There exists $C>0$ such that for any $x=(x_{\sigma})_{\sigma} \in T_F$, we have
\begin{align}
\sum_{u \in U_F}\prod_{\sigma \in S_F}e^{-x_{\sigma}u_{\sigma}} \leq C N_{F/\mathbb Q}(x)^{-\frac{1}{4}}.
\end{align}
See Section \ref{heeg obj} for the definition of $U_F \subset T_F$.
\end{enumerate}
\end{lem}
\begin{proof}
(1) is elementary and we omit the proof. To prove (2) we use the following well-known asymptotic behavior of the $K$-Bessel function:
\begin{align}\label{asymp k}
\lim_{x \to \infty} \sqrt{x} K_{s}(x)e^x = \sqrt{\frac{\pi}{2}}, \quad
\lim_{x \to 0}K_s(x)x^{s} = 2^{s-1}\Gamma(s), \quad(s\in \mathbb C, \re{s}>0)
\end{align}
Now, since $|K_s(x)x^{-s}|\leq K_{\re s}(x)x^{-\re s}$, we may assume $-R \leq s \leq R$. Then, by the integral expression (\ref{K-bessel}) of $K_s(x)$, we have
\begin{align}
K_s(x) 
&\leq 
\frac{1}{2} \int_{0}^{1} e^{-\frac{1}{2}x(u+\frac{1}{u})}u^{-R}\frac{du}{u}
+\frac{1}{2} \int_{1}^{\infty} e^{-\frac{1}{2}x(u+\frac{1}{u})}u^R\frac{du}{u}\\
&\leq K_{-R}(x)+K_R(x)= 2K_R(x).
\end{align}
Therefore, we have 
\begin{align}
K_{s}(x)x^{-s}e^{x/2}x^{10R} \leq 2K_R(x)x^{10R-s}e^{x/2}.
\end{align}
By the asymptotic formulas, the right hand side can be bounded uniformly in $s$ for $-R \leq s \leq R$. This proves (2). 
To prove (3), first observe that 
\begin{align}
\sum_{\sigma} x_{\sigma}u_{\sigma} &\geq \sum_{\substack{\sigma\\ x_{\sigma}u_{\sigma}\geq 1}} \log(x_{\sigma}u_{\sigma})
\geq \frac{1}{4}\sum_{\substack{\sigma\\ x_{\sigma}u_{\sigma}\geq 1}} 2n_{\sigma} \log(x_{\sigma}u_{\sigma})\\
&= \frac{1}{4}(\log(N_{F/\mathbb Q}(x))+ \sum_{\sigma} n_{\sigma} |\log(x_{\sigma}u_{\sigma})|),
\end{align}
where we use $N_{F/\mathbb Q}(u)=1$. Thus we have
\begin{align}
\sum_{u \in U_F}\prod_{\sigma}e^{-x_{\sigma}u_{\sigma}} 
\leq N_{F/\mathbb Q}(x)^{-\frac{1}{4}} \sum_{u \in U_F} \prod_{\sigma}e^{-\frac{1}{4}(n_{\sigma}|\log(u_{\sigma})+\log (x_{\sigma})|)}.
\end{align}
Then, we easily see that there exists $C>0$ which is independent of $x$ such that 
\begin{align}
\sum_{u \in U_F} \prod_{\sigma}e^{-\frac{1}{4}(n_{\sigma}|\log(u_{\sigma})+\log (x_{\sigma})|)}\leq C,
\end{align}
using the fact that $U_F$ is a lattice in $T_{F/\mathbb Q}$.
This proves (3).
\end{proof}

Let $R>n\geq 1$. Then, by Lemma \ref{asymp lem}, there exist $C_1, C_2>0$ such that for $|\re{ns-j-1}|\leq R$, the second row of (\ref{eqn459}) can be bounded as 
\begin{align}
&\sum_{\substack{\mathfrak n \in \mathscr A}} \sum_{\nu \in (\Lambda^{(j+1)})_{\mathfrak n^{-1}}}\left| (N\mathfrak n)^{ns-j-1} \sigma_{1+j-ns}(\mathfrak n, \chi)
\prod_{\sigma}\frac{K_{\frac{n_{\sigma}}{2}(ns-j-1)}(2 \pi ||\nu g^{(j+1)}_{\sigma}||)}{||\nu g^{(j+1)}_{\sigma}||^{\frac{n_{\sigma}}{2}(ns-j-1)}}\right|\\
\leq &C_1 \sum_{\mathfrak n \in \mathscr A}\sum_{\nu \in (\Lambda^{(j+1)})_{\mathfrak n^{-1}}}
(N\mathfrak n)^{3R}\prod_{\sigma}||\nu g^{(j+1)}_{\sigma}||^{-5n_{\sigma}R}e^{-\pi||\nu g^{(j+1)}_{\sigma}||}\\
\leq & C_1 \sum_{\mathfrak n \in \mathscr A}\sum_{u\in \mathcal O_F^{\times}}\sum_{\nu \in \mathcal O_F^{\times} \bs (\Lambda^{(j+1)})_{\mathfrak n^{-1}}}
(N\mathfrak n)^{3R}\prod_{\sigma}||u \nu g^{(j+1)}_{\sigma}||^{-5n_{\sigma}R}e^{-\pi||u\nu g^{(j+1)}_{\sigma}||}\\
\leq & C_2 \sum_{\mathfrak n \in \mathscr A}\sum_{\nu \in \mathcal O_F^{\times} \bs (\Lambda^{(j+1)})_{\mathfrak n^{-1}}}
(N\mathfrak n)^{3R}\prod_{\sigma}||\nu g^{(j+1)}_{\sigma}||^{-5n_{\sigma}R-\frac{1}{4}n_{\sigma}}. \label{eqn 468}
\end{align}
Let us fix $\mathfrak n \in \mathscr A$. Then, (\ref{eqn 468}) can be rewritten as
\begin{align}
&C_3 \sum_{x\in \mathcal O_F^{\times} \bs \mathfrak n^{-1}} |N_{F/\mathbb Q}(x)|^{3R-5R-\frac{1}{4}}
\sum_{\nu \in \mathcal O_F^{\times} \bs (\Lambda^{(j+1)})_{\mathfrak n^{-1}}}\prod_{\sigma}||\nu g^{(j+1)}_{\sigma}||^{-n_{\sigma}(5R+\frac{1}{4})}\\
=&C_3 (N\mathfrak n)^{2R+\frac{1}{4}} \zeta_F(\mathfrak n^{-1}, 2R+\frac{1}{4})
\sum_{\nu \in \mathcal O_F^{\times} \bs (\Lambda^{(j+1)})_{\mathfrak n^{-1}}}
\prod_{\sigma}||\nu g^{(j+1)}_{\sigma}||^{-n_{\sigma}(5R+\frac{1}{4})},
\end{align} 
where $C_3:=C_2(N\mathfrak n)^{3R}$. 
By Proposition \ref{prop eis ser} (1) applied to the parabolic data $\mathfrak n^{-1} \hookrightarrow \Lambda^{(j+1)}$, the last infinite series converges. This completes the proof of Theorem \ref{phi psi conti}. \qed 

\section{Residue formula and Kronecker's limit formula}\label{res and limit formula}
In this section, we compute the residue and the constant term of our Eisenstein series $E_{L,[\mathfrak a]}(z,s)$ and the ``relative'' partial zeta function $\zeta_{E/F,[\mathfrak a]}(\mathfrak A^{-1},s)$ at $s=1$. 

We keep the notations in Section \ref{fourier exp}. That is, $F$ is a number field of degree $d$, $L\subset F^n$ $(n\geq 2)$ is an $\mathcal O_F$-lattice of the form $L=\mathfrak a_1 \oplus \cdots \oplus \mathfrak a _n$, and $E_{L,[\mathfrak a]}(z,s)$ is the Eisenstein series associated to a parabolic data $\data{\mathfrak a}{L}$, etc.
In the following, for any meromorphic function $f$ on $\mathbb C$, and $a \in \mathbb C$, we denote by $f^{(-1)}(a)$ (resp. $f^{(0)}(a)$) the residue (resp. constant term) of $f(s)$ at $s=a$.

\subsection{Residue formula for $E_{L,[\mathfrak a]}(z,s)$}

Recall that by the Fourier expansion formula (Theorem \ref{fourier exp thm}), we have
\begin{align}\label{eqn 51}
\frac{E_{L,[\mathfrak a]}(z,s)}{\zeta_F(\mathfrak a^{-1},ns)}= \sum_{j=0}^{n-2}(\varPhi_j(z,s)+\varPsi_j(z,s))+\varPhi_{n-1}(z,s).
\end{align}
Then, by Theorem \ref{phi psi conti}, $ \sum_{j=0}^{n-2}(\varPhi_j(z,s)+\varPsi_j(z,s))$ is holomorphic at $s=1$.
On the other hand, by Corollary \ref{z_a cor}, we have
\begin{align}\label{eqn 52}
\varPhi_{n-1}(z,s)=c_{n-1}(z,s) \frac{1}{h_F}\sum_{\chi \in \Hom(Cl_F, \mathbb C^{\times})} \chi( [\mathfrak a_1\mathfrak a^{-1}]) \frac{L(ns-n+1,\chi)}{L(ns,\chi)}.
\end{align}
Therefore we get the following.
\begin{thm}[Residue formula for $E_{L,[\mathfrak a]}(z,s)$]\label{res formula eis}
The Eisenstein series $E_{L,[\mathfrak a]}(z,s)$ has a simple pole at $s=1$, and we have 
\begin{align}
&E_{L,[\mathfrak a]}^{(-1)}(z,1) = c_{n-1}(z,1)\frac{\kappa_F}{n h_F} \frac{\zeta_F(\mathfrak a^{-1},n)}{\zeta_F(n)} \nonumber\\
&= (N\mathfrak a_1\cdots N\mathfrak a_n)^{-1} 
\left(\frac{2^{r_2(F)}\pi^{\frac{d}{2}}}{\sqrt{|d_F|}}\right)^n\frac{2^{r_1(F)}}{\Gamma_F(n)}
\frac{R_F}{n w_F} 
\frac{\zeta_F(\mathfrak a^{-1},n)}{\zeta_F(n)},
\end{align}
where $\zeta_F(s)$ is the Dedekind zeta function of $F$, and $\kappa_F = \zeta_F^{(-1)}(1)$ is its residue at $s=1$.
\end{thm}
\begin{rmk}
Note that $c_{n-1}(z,1)$ and hence $E_{L,[\mathfrak a]}^{(-1)}(z,1)$ are independent of $z$. We put $\mathbf c_{n-1}:=c_{n-1}(z,1)$ for simple. 
\end{rmk}

\begin{proof}
This follows from the fact that $L(s,\chi)$ is holomorphic at $s=1$ unless $\chi$ is trivial, in which case $L(s,\chi)=\zeta_F(s)$ has a simple pole. 
\end{proof}

\subsection{Kronecker's limit formula for $E_{L,[\mathfrak a]}(z,s)$}
First recall the classical case where $F=\mathbb Q$, $n=2$, $L=\mathbb Z \oplus \mathbb Z$, and $\mathfrak a=\mathbb Z$ (cf. Example \ref{phi psi ex}). In this case, we have
\begin{align}
\varPhi_0^{(0)}(z,1)&+\varPsi_0^{(0)}(z,1) = \varPhi_0(z,1)+\varPsi_0(z,1) = -\frac{6}{\pi} \log|\eta (z)|^2,\label{eqn 54}\\
\varPhi_1^{(0)}(z,1)& = 
\frac{6}{\pi} \left(\gamma + \frac{1}{2}\left(\psi \left(\frac{1}{2} \right)-\psi(1)-\log y_1 - 2\frac{\zeta_{\mathbb Q}^{(1)}(2)}{\zeta_{\mathbb Q}(2)} \right) \right)\\
&=\frac{6}{\pi} (\gamma -\log 2-\log \sqrt{y_1}) - \frac{6}{\pi} \frac{\zeta_{\mathbb Q}^{(1)}(2)}{\zeta_{\mathbb Q}(2)},
\end{align}
where $\eta(z)$ is the Dedekind eta function, $\gamma = \zeta_{\mathbb Q}^{(0)}(1)$ is Euler's constant, and $\psi(s) := \frac{\Gamma'(s)}{\Gamma(s)}$ is the digamma function, that is, the logarithmic derivative of the gamma function.

\begin{dfn}\label{limit formula dfn}
\begin{enumerate}[{\rm (1)}]
\item We define 
\begin{align}
H_{L,[\mathfrak a]}(z):= \zeta_F(\mathfrak a^{-1},n)\sum_{j=0}^{n-2}(\varPhi_j(z,1)+\varPsi_j(z,1)), \quad\text{for }z\in \mathfrak h_F^n.
\end{align}
\item Let $k/\mathbb Q$ be a number field. 
For a character $\chi \in \Hom(Cl_k,\mathbb C^{\times})$, let $L(s,\chi)$  be the Hecke $L$-function associated to the character $\chi$. Then we set 
\begin{align}
\gamma_k(\chi):= L^{(0)} (1,\chi).
\end{align}
Moreover, let $\Gamma_k(s)$ be the gamma function over $k$ (\ref{abs gamma}). We define the digamma function $\psi_k(s)$ over $k$ as
\begin{align}
\psi_k(s) := \frac{d}{ds} \log \Gamma_k(s) = \frac{\Gamma_k '(s)}{\Gamma_k(s)} = \sum_{\sigma \in S_k} \frac{n_{\sigma}}{2}\psi \left(\frac{n_{\sigma}s}{2}\right).
\end{align}
\end{enumerate}
\end{dfn}

\begin{thm}[Kronecker's limit formula for $E_{L,[\mathfrak a]}(z,s)$]\label{limit formula eis}
We have 
\begin{multline}
E_{L,[\mathfrak a]}^{(0)}(z,1) = H_{L,[\mathfrak a]}(z) + \mathbf c_{n-1} \frac{1}{h_F}\sum_{\chi \in \Hom(Cl_F,\mathbb C^{\times})} \chi ([\mathfrak a_1 \mathfrak a^{-1}])\frac{\zeta_F(\mathfrak a^{-1},n)}{L(n,\chi)} \gamma_F(\chi)  \\
+\mathbf c_{n-1} \frac{\kappa_F}{h_F} \frac{\zeta_F(\mathfrak a^{-1},n)}{\zeta_F(n)} 
\Bigg(\psi_F(1)-\psi_F(n) -\log(N\mathfrak a_1) -\sum_{k=1}^{n-1}\frac{n-k}{n}\log (N_{F/\mathbb Q}(y_k))\Bigg)  \\
+ \mathbf c_{n-1} \frac{\kappa_F}{h_F} \frac{\zeta_F(\mathfrak a^{-1},n)}{\zeta_F(n)} 
\Bigg( \frac{\zeta_F^{(1)}(\mathfrak a^{-1},n)}{\zeta_F(\mathfrak a^{-1},n)} - \frac{\zeta_F^{(1)}(n)}{\zeta_F(n)} \Bigg).
\end{multline}
\end{thm}
\begin{proof}
This follows from the Fourier expansion formula (\ref{eqn 51}) (Theorem \ref{fourier exp thm}), (\ref{eqn 52}), and Theorem \ref{res formula eis}.
\end{proof}

\begin{cor}[Automorphy of the function $H_{L,[\mathfrak a]}$]\label{automorphy H}
The function 
\begin{align}
H^*_{L,[\mathfrak a]}(z):= H_{L,[\mathfrak a]}(z)- \mathbf c_{n-1} \frac{\kappa_F}{h_F} \frac{\zeta_F(\mathfrak a^{-1},n)}{\zeta_F(n)} \sum_{k=1}^{n-1}\frac{n-k}{n}\log (N_{F/\mathbb Q}(y_k))
\end{align}
is an automorphic function on $\mathfrak h_F^n$, that is, we have 
$H^*_{L,[\mathfrak a]}(\gamma z)=H^*_{L,[\mathfrak a]}(z)$ for all $\gamma \in \Gamma_L$.
\end{cor}
\begin{proof}
Since $E_{L,[\mathfrak a]}(z,s)$ is an automorphic function (Proposition \ref{prop eis ser}), $E_{L,[\mathfrak a]}^{(0)}(z,1) $ is also an automorphic function on $\mathfrak h_F^n$. 
Then, by Theorem \ref{limit formula eis}, we obtain
$H^*_{L,[\mathfrak a]}(\gamma z)-H^*_{L,[\mathfrak a]}(z)=E_{L,[\mathfrak a]}^{(0)}(\gamma z,1) -E_{L,[\mathfrak a]}^{(0)}(z,1)=0 \text{ for all }\gamma \in \Gamma_L.$
\end{proof}

\begin{rmk}
By the identity (\ref{eqn 54}), we see that the function $H_{L,[\mathfrak a]}(z)$ gives a generalization of the function $-\frac{6}{\pi} \log|\eta (z)|^2$. In fact, in the case where $n=2$, the sum $\sum_{\mathscr A\in Cl_F}H_{L,\mathscr A}(z)$ coincides with the function $h_F(z,\mathfrak a,\mathfrak b)$ considered by Yamamoto \cite[Theorem 2.5.1]{yamamoto08} up to some constant factors, and Corollary \ref{automorphy H} gives a generalization of  \cite[Corollary 2.5.2]{yamamoto08}.
\end{rmk}

\subsection{Application to $\zeta_{E/F,[\mathfrak a]}(\mathfrak A^{-1},s)$}

Let the notations be as in Section \ref{heeg obj} and Section \ref{rel int formula}. That is, $E/F$ is an extension of number fields of degree $n\geq 2$, and $\mathfrak A \subset E$ is a fractional $\mathcal O_E$-ideal. In this section, we take a basis $w:={}^t\!(w_1,\dots, w_n)$ of $E$ over $F$ so that the lattice $L \subset F^n$ corresponding to $\mathfrak A$ via the isomorphism $w: F^n \overset{\sim}{\rightarrow} E$ is of the form $L=\mathfrak a_1\oplus \cdots \oplus \mathfrak a_n$ for some anti-integral $\mathcal O_F$-ideals $\mathfrak a_i$ ($1\leq i\leq n$). 
Note that this is always possible by Proposition \ref{str thm}. Then, $\varpi :T_{E/F}\rightarrow \mathfrak h_F^n$ is the Heegner object associated to $w$. 
Let $\mathfrak a$ be an anti-integral $\mathcal O_F$-ideal, and let $E_{L,[\mathfrak a]}(z,s)$ be the Eisenstein series associated to the parabolic data $\data{\mathfrak a}{L}$.

Combining Theorem \ref{hecke's int formula}, Theorem \ref{fourier exp thm}, Theorem \ref{res formula eis}, and Theorem \ref{limit formula eis},
we obtain the following properties of $\zeta_{E/F,[\mathfrak a]}(\mathfrak A^{-1},s)$.

\begin{thm}\label{rel partial zeta thm}
\begin{enumerate}[{\rm (1)}]
\item The relative partial zeta function $\zeta_{E/F,[\mathfrak a]}(\mathfrak A^{-1},s)$ can be continued meromorphically to whole $s\in \mathbb C$, and has a simple pole at $s=1$.
\item (Residue formula) We have 
\begin{align}
\zeta_{E/F,[\mathfrak a]}^{(-1)}(\mathfrak A^{-1},1)
&= \frac{2^{r_1(E)}(2\pi)^{r_2(E)}R_E}{w_E \sqrt{|d_E|}} \frac{\zeta_F(\mathfrak a^{-1},n)}{\zeta_F(n)}  \\
&=\frac{\kappa_E}{h_E} \frac{\zeta_F(\mathfrak a^{-1},n)}{\zeta_F(n)}, \label{rel zeta res formula}
\end{align}
where $\kappa_E := \zeta_E^{(-1)}(1)$ is the residue of the Dedekind zeta function $\zeta_E(s)$ of $E$. 
\item (Kronecker's limit formula) We have
\begin{multline}
\zeta_{E/F,[\mathfrak a]}^{(0)}(\mathfrak A^{-1},1)= 
n \mathbf c_{n-1}^{-1}  \frac{h_F }{h_E } \frac{\kappa_E}{\kappa_F}  \Bigg( \frac{1}{R_{E/F}} \int_{T_{E/F}/U_{E/F}} H^*_{L,[\mathfrak a]}(\overline{\varpi}(t))\, \dtimes t_{E/F} \\
+ \frac{1}{h_F}\sum_{\chi} \chi([\mathfrak a_1\mathfrak a^{-1}]) \frac{\zeta_F(\mathfrak a^{-1},n)}{L(n,\chi)} \gamma_F(\chi) \Bigg) \\
 +\frac{\kappa_E}{h_E} \frac{\zeta_F(\mathfrak a^{-1},n)}{\zeta_F(n)} 
\Bigg((nr_F-r_E)\log 2 + \log \frac{|d_E|^{\frac{1}{2}}}{|d_F|^{\frac{n}{2}}} 
+ \log \frac{N \mathfrak a_1^{n}}{N\mathfrak a_1\cdots N \mathfrak a_n}\\
 -n \frac{\zeta_F^{(1)}(\mathfrak a^{-1},n)}{\zeta_F(\mathfrak a^{-1},n)} +n \frac{\zeta_F^{(1)}(n)}{\zeta_F(n)}\Bigg).
\end{multline}
\end{enumerate}
\end{thm}

\begin{proof}
(1) follows from Theorem \ref{hecke's int formula} and Theorem \ref{fourier exp thm}. 
To prove (2) and (3), we use the following identity:
\begin{align}
|\Delta_w|^{\frac{1}{2}}(=|N_{F/\mathbb Q}(\det W)|)= \frac{2^{nr_2(F)-r_2(E)}|d_E|^{\frac{1}{2}}N\mathfrak A}{|d_F|^{\frac{n}{2}} N\mathfrak a_1\cdots N\mathfrak a_n}.
\end{align}
Indeed, by the definition of $W$ and the normalization of the Haar measures $dx_E$, $dx_F$ (Section \ref{haar meas}), we see
\begin{align}
|d_E|^{\frac{1}{2}}N\mathfrak A
&= vol(E_{\infty}/\mathfrak A) = 2^{r_2(E)-nr_2(F)}|\Delta_w|^{\frac{1}{2}} vol(F_{\infty}^n/\mathfrak a_1\oplus \dots \oplus \mathfrak a_n) \nonumber \\
&= 2^{r_2(E)-nr_2(F)}|\Delta_w|^{\frac{1}{2}} |d_F|^{\frac{n}{2}} N\mathfrak a_1\cdots N\mathfrak a_n
\end{align}
Then, (2) and (3) follows from Theorem \ref{res formula eis} and Theorem \ref{limit formula eis} combined with Theorem \ref{hecke's int formula}. 
\end{proof}

\subsection*{Acknowledgements}
I would like to express my deepest gratitude to Professor Takeshi Tsuji for the constant encouragement and valuable comments during the study. 
I would also like to thank Kazuki Hiroe for pointing to me about the Hecke's result in the case of general number fields \cite{hecke17}. 
Part of this paper is written during my stay at the Max Planck Institute for Mathematics in 2018. The author is supported by JSPS Overseas Challenge Program for Young Researchers.

{\small
\bibliographystyle{alpha}
\bibliography{hecke's_integral_formula}

\begin{bibdiv}
\begin{biblist}

\bib{asai70}{article}{
      author={Asai, Tetsuya},
       title={On a certain function analogous to {${\rm log}|{\eta }(z)|$}},
        date={1970},
        ISSN={0027-7630},
     journal={Nagoya Math. J.},
      volume={40},
       pages={193\ndash 211},
         url={http://projecteuclid.org/euclid.nmj/1118798133},
}

\bib{bekki17}{article}{
      author={Bekki, Hohto},
       title={On periodicity of geodesic continued fractions},
        date={2017},
        ISSN={0022-314X},
     journal={J. Number Theory},
      volume={177},
       pages={181\ndash 210},
         url={http://dx.doi.org/10.1016/j.jnt.2017.01.003},
}

\bib{borel77}{article}{
      author={Borel, Armand},
       title={Cohomologie de {${\rm SL}\sb{n}$} et valeurs de fonctions zeta
  aux points entiers},
        date={1977},
     journal={Ann. Scuola Norm. Sup. Pisa Cl. Sci. (4)},
      volume={4},
      number={4},
       pages={613\ndash 636},
         url={http://www.numdam.org/item?id=ASNSP_1977_4_4_4_613_0},
}

\bib{bourbakicommalg}{book}{
      author={Bourbaki, Nicolas},
       title={Commutative algebra. {C}hapters 1--7},
      series={Elements of Mathematics (Berlin)},
   publisher={Springer-Verlag, Berlin},
        date={1998},
        ISBN={3-540-64239-0},
        note={Translated from the French, Reprint of the 1989 English
  translation},
}

\bib{bumpgoldfeld84}{incollection}{
      author={Bump, Daniel},
      author={Goldfeld, Dorian},
       title={A {K}ronecker limit formula for cubic fields},
        date={1984},
   booktitle={Modular forms ({D}urham, 1983)},
      series={Ellis Horwood Ser. Math. Appl.: Statist. Oper. Res.},
   publisher={Horwood, Chichester},
       pages={43\ndash 49},
}

\bib{costafriedman91}{article}{
      author={Costa, Antone},
      author={Friedman, Eduardo},
       title={Ratios of regulators in totally real extensions of number
  fields},
        date={1991},
        ISSN={0022-314X},
     journal={J. Number Theory},
      volume={37},
      number={3},
       pages={288\ndash 297},
         url={http://dx.doi.org/10.1016/S0022-314X(05)80044-7},
}

\bib{epstein03}{article}{
      author={Epstein, Paul},
       title={Zur {T}heorie allgemeiner {Z}etafunctionen},
        date={1903},
        ISSN={0025-5831},
     journal={Math. Ann.},
      volume={56},
      number={4},
       pages={615\ndash 644},
         url={https://doi.org/10.1007/BF01444309},
}

\bib{goldfeld06}{book}{
      author={Goldfeld, Dorian},
       title={Automorphic forms and {$L$}-functions for the group {${\rm
  GL}(n,\bold R)$}},
      series={Cambridge Studies in Advanced Mathematics},
   publisher={Cambridge University Press, Cambridge},
        date={2006},
      volume={99},
        ISBN={978-0-521-83771-2; 0-521-83771-5},
         url={http://dx.doi.org/10.1017/CBO9780511542923},
        note={With an appendix by Kevin A. Broughan},
}

\bib{harder82}{incollection}{
      author={Harder, G.},
       title={Period integrals of {E}isenstein cohomology classes and special
  values of some {$L$}-functions},
        date={1982},
   booktitle={Number theory related to {F}ermat's last theorem ({C}ambridge,
  {M}ass., 1981)},
      series={Progr. Math.},
      volume={26},
   publisher={Birkh\"auser Boston, Boston, MA},
       pages={103\ndash 142},
         url={https://doi.org/10.1007/978-1-4899-6699-5_9},
}

\bib{hecke17}{article}{
      author={Hecke, Erich},
       title={{\"U}ber die {K}roneckersche {G}renzformel f{\"u}r reelle
  quadratische {K}{\"o}rper und die {K}lassenzahl relativ-{A}belscher
  {K}{\"o}rper},
        date={1917},
     journal={Verhandl. Naturforsch. Ges. Basel},
      volume={28},
       pages={363\ndash 372},
}

\bib{hiroeoda08}{article}{
      author={Hiroe, Kazuki},
      author={Oda, Takayuki},
       title={Hecke-{S}iegel's pull-back formula for the {E}pstein zeta
  function with a harmonic polynomial},
        date={2008},
        ISSN={0022-314X},
     journal={J. Number Theory},
      volume={128},
      number={4},
       pages={835\ndash 857},
         url={http://dx.doi.org/10.1016/j.jnt.2007.08.010},
}

\bib{jorgenson-lang99}{article}{
      author={Jorgenson, Jay},
      author={Lang, Serge},
       title={Hilbert-{A}sai {E}isenstein series, regularized products, and
  heat kernels},
        date={1999},
        ISSN={0027-7630},
     journal={Nagoya Math. J.},
      volume={153},
       pages={155\ndash 188},
         url={http://projecteuclid.org/euclid.nmj/1114630824},
}

\bib{langlands76}{book}{
      author={Langlands, Robert~P.},
       title={On the functional equations satisfied by {E}isenstein series},
      series={Lecture Notes in Mathematics, Vol. 544},
   publisher={Springer-Verlag, Berlin-New York},
        date={1976},
}

\bib{liumasri15}{article}{
      author={Liu, Sheng-Chi},
      author={Masri, Riad},
       title={A {K}ronecker limit formula for totally real fields and
  arithmetic applications},
        date={2015},
        ISSN={2363-9555},
     journal={Res. Number Theory},
      volume={1},
       pages={Art. 8, 20},
         url={https://doi.org/10.1007/s40993-015-0009-3},
}

\bib{sarnak82}{article}{
      author={Sarnak, Peter},
       title={Class numbers of indefinite binary quadratic forms},
        date={1982},
        ISSN={0022-314X},
     journal={J. Number Theory},
      volume={15},
      number={2},
       pages={229\ndash 247},
         url={http://dx.doi.org/10.1016/0022-314X(82)90028-2},
}

\bib{sczech93}{article}{
      author={Sczech, Robert},
       title={Eisenstein group cocycles for {${\rm GL}_n$} and values of
  {$L$}-functions},
        date={1993},
        ISSN={0020-9910},
     journal={Invent. Math.},
      volume={113},
      number={3},
       pages={581\ndash 616},
         url={https://doi.org/10.1007/BF01244319},
}

\bib{siegel80}{book}{
      author={Siegel, Carl~Ludwig},
       title={Advanced analytic number theory},
     edition={Second},
      series={Tata Institute of Fundamental Research Studies in Mathematics},
   publisher={Tata Institute of Fundamental Research, Bombay},
        date={1980},
      volume={9},
}

\bib{yamamoto08}{article}{
      author={Yamamoto, Shuji},
       title={Hecke's integral formula for relative quadratic extensions of
  algebraic number fields},
        date={2008},
        ISSN={0027-7630},
     journal={Nagoya Math. J.},
      volume={189},
       pages={139\ndash 154},
         url={http://projecteuclid.org/euclid.nmj/1205156912},
}

\bib{yoshida03}{book}{
      author={Yoshida, Hiroyuki},
       title={Absolute {CM}-periods},
      series={Mathematical Surveys and Monographs},
   publisher={American Mathematical Society, Providence, RI},
        date={2003},
      volume={106},
        ISBN={0-8218-3453-3},
         url={http://dx.doi.org/10.1090/surv/106},
}

\end{biblist}
\end{bibdiv}

}

{\small Graduate School of Mathematical Sciences, The University of Tokyo, 3-8-1 Komaba, Meguro, Tokyo, 153-8914 Japan

{\it Email}:  bekki@ms.u-tokyo.ac.jp}


\end{document}